\documentclass[a4paper,12pt]{amsart}
\usepackage{amssymb, enumerate,psfrag,graphicx,amsfonts,amsrefs,amsthm,mathrsfs,amsmath,amscd,version,graphicx}
\usepackage{tikz,cleveref,caption}

\usetikzlibrary{graphs,graphs.standard}

\usepackage{xcolor}
\usepackage{tikz-cd}
\usepackage{tikz}
\usetikzlibrary{arrows}
\tikzset{negated/.style={
        decoration={markings,
            mark= at position 0.5 with {
                \node[transform shape] (tempnode) {$\backslash$};
            }
        },
        postaction={decorate}
    }
}
\tikzset{
    vertex/.style={draw,circle,thick, inner sep=1.5pt,minimum size=6pt,
    },
    edge/.style={thick},
    dedge/.style = {->,> = latex',thick}
}
\definecolor{asparagus}{rgb}{0.53, 0.66, 0.42}
\usetikzlibrary{decorations.markings}
\usetikzlibrary{decorations.text}
\Crefname{equation}{Eq.}{Eqs.}

\newtheorem{theo}{Theorem}[section]
\newtheorem{theorem}[theo]{Theorem}
\newtheorem{lemma}[theo]{Lemma}
\newtheorem{proposition}[theo]{Proposition}

\newtheorem{definition}[theo]{Definition}

\theoremstyle{definition}
\newtheorem{remark}[theo]{Remark}
\newtheorem{example}[theo]{Example}

\numberwithin{equation}{section}

\begin{document}

\title{Godsil-McKay switchings for gain graphs}

\author[M. Cavaleri]{Matteo Cavaleri}
\address{Matteo Cavaleri, Universit\`{a} degli Studi Niccol\`{o} Cusano - Via Don Carlo Gnocchi, 3 00166 Roma, Italia}
\email{matteo.cavaleri@unicusano.it   \textrm{(Corresponding Author)}}

\author[A. Donno]{Alfredo Donno}
\address{Alfredo Donno, Universit\`{a} degli Studi Niccol\`{o} Cusano - Via Don Carlo Gnocchi, 3 00166 Roma, Italia}
\email{alfredo.donno@unicusano.it}

\author[S. Spessato]{Stefano Spessato}
\address{Stefano Spessato, Universit\`{a} degli Studi Niccol\`{o} Cusano - Via Don Carlo Gnocchi, 3 00166 Roma, Italia}
\email{stefano.spessato@unicusano.it}

\begin{abstract}
We introduce a switching operation, inspired by the Godsil-McKay switching, in order to obtain pairs of $G$-cospectral gain graphs, that are gain graphs cospectral with respect to every representation of the gain group $G$. For instance, for two signed graphs, this notion of cospectrality is equivalent to the cospectrality of their signed adjacency matrices together with the cospectrality of their underlying graphs. Moreover, we introduce another more flexible switching in order to obtain pairs of gain graphs cospectral with respect to some fixed unitary representation. Many existing notions of spectrum for  graphs and gain graphs are indeed special cases of these spectra associated with particular representations, therefore our construction recovers the classical Godsil-McKay switching and the Godsil-McKay switching for signed and complex unit gain graphs. As in the classical case, not all gain graphs are suitable for these switchings: we analyze the relationships between the properties that make the graph suitable for the one or the other switching. Finally we apply our construction in order to define a Godsil-McKay switching for the right spectrum of  quaternion unit gain graphs.

\end{abstract}

\maketitle

\begin{center}
\end{center}

\section{Introduction}
Spectral graph theory is the investigation of graphs via linear algebra. It mainly consists in studying invariants such as the spectrum, the spectral radius, the index, and many others, of matrices representing  graphs, such as the \emph{adjacency matrix}, the \emph{Laplacian matrix}, and other generalizations.  Many geometric and combinatorial properties of a graph can be deduced from these spectral invariants: the theory is rich and well established  \cite{cve} and it is omnipresent in various fields, from the most theoretical to the most applied ones \cite{cveapp}.
\\ \indent For example the first time anyone wondered about the existence of \emph{pairs of nonisomorphic cospectral graphs}, it was in the world of chemistry  \cite{gp} in the 1950s. In fact, at that time, the question was whether graphs were \emph{determined by their spectrum} (see, for instance,  \cite{deter,deter2} for more historic details and further developments). Cospectral graphs have attracted a lot of attention for many reasons. In particular, they are a useful tool to outline the limits of the information that can be drawn from the spectrum and  therefore the limits of spectral graph theory itself. A simple but emblematic example of a pair of cospectral graphs is the one depicted in Fig.~\ref{fig:-1}, firstly given in \cite{cve0}  and called \emph{Saltire pair} in \cite{deter}, due to the similarity to the Scottish flag.
One of the two graphs is connected and one is not: therefore connectivity, at least in general, cannot be deduced from the adjacency spectrum.
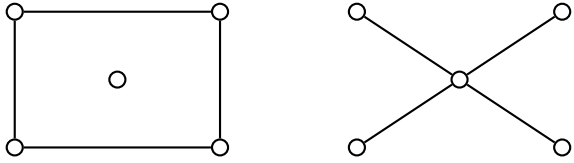
\begin{figure}
\centering
\begin{tikzpicture}[scale=0.9]


\node[vertex]  (1) at (0,0){};

\node[vertex]  (2) at (3,0){};

\node[vertex]  (3) at (3,2){};
\node[vertex]  (4) at (0,2){};
\node[vertex]  (5) at (1.5,1){};


\draw[edge] (1) -- (2);
\draw[edge] (2) -- (3);
\draw[edge] (3) -- (4);
\draw[edge] (1) -- (4);


\node[vertex]  (1) at (0+5,0){};

\node[vertex]  (2) at (3+5,0){};

\node[vertex]  (3) at (3+5,2){};
\node[vertex]  (4) at (0+5,2){};
\node[vertex]  (5) at (1.5+5,1){};


\draw[edge] (1) -- (5);
\draw[edge] (2) -- (5);
\draw[edge] (3) -- (5);
\draw[edge] (4) -- (5);

\end{tikzpicture}
\caption{Two cospectral graphs.}\label{fig:-1}
\end{figure}
\\\indent But the real leap from an anecdotal research to a systematic study in the research of cospectral pairs happened when also a quantitative, probabilistic investigation of the spectral determination began. A remarkable achievement in this direction was the result that \emph{almost all} trees are cospectral \cite{tree}.
Many methods, operations, routines have been proposed to obtain cospectrality: one of the most refined results, allowing the consolidation of various known techniques into a single construction, is the so-called \emph{Godsil-McKay switching} \cite{GoMc}.

The ubiquity of graph theory is due to the fact that a graph is a model for a system of interactions. When these interactions may be positive or negative, then the related natural object is instead a \emph{signed graph}, i.e., a graph whose edges may be positive or negative. Formally, a signed graph is a pair $(\Gamma,\sigma)$ where $\Gamma$ is the \emph{underlying graph} and $\sigma$ is the \emph{signature}, that is a map assigning a sign to each adjacency.  In order to model interactions that are even more varied, there exists a further generalization, given by the notion of \emph{gain graph}. This  is a pair $(\Gamma,\psi)$ where still $\Gamma$ is the underlying graph and $\psi$ is the \emph{gain function}, assigning to every ordered pair of adjacent vertices $u,v$ of $\Gamma$ an element $\psi(u,v)$, called \emph{gain}, from a group $G$, in such a way that  $\psi(v,u)=\psi(u,v)^{-1}$.
Signed graphs are then particular gain graphs with $G = \mathbb{T}_2=\{\pm 1\}$. As graphs are studied up to isomorphism, signed and gain graphs are studied up to \emph{switching isomorphism} \cite{zasgraph, zasign, zaslavsky1}. When the  product of the gains (signs) along any closed walk is the unit element of $G$, the gain (signed) graph is said to be \emph{balanced} \cite{Harary}. Equivalently, a gain graph is balanced if it is switching isomorphic to its underlying graph endowed with the trivial gain function. For a regularly
updated bibliography  on signed and gain graphs we refer to \cite{zasbib}.
\\ \indent Spectral graph theory was generalized in a very natural way to signed graphs, which  have $\{0,\pm 1\}$-valued adjacency matrices \cite{zasmat}.
The  spectrum of a signed graph is even able to capture  characteristics  that are specific of signed graphs, such as balance \cite{acharya}. Moreover, a signed version of the Godsil-McKay switching  has been proposed in \cite{cos}.
\\ \indent The generalization of the spectral theory to gain graphs is less immediate and univocal. Many particular cases have been considered in literature: first of all, the spectral theory of  \emph{complex unit gain graphs} \cite{reff1}, that are gain graphs whose gain group is $\mathbb T=\{z\in\mathbb C: |z|=1\}$. Clearly, also in this case,
complex matrices naturally arise and a spectral characterization of balance was proved  \cite{adun}, as well as an analogue of the Godsil-McKay switching \cite{gm}. Also the spectral theory for
 \emph{quaternion unit gain graphs} was recently investigated  \cite{quat}, thanks to the existence of several notions of spectrum for quaternion valued matrices \cite{brenner,Lee,Spec,zhang}. In  \cite{quat} many of the classical results for signed and complex unit gain graphs are generalized to quaternion unit gain graphs. One of the byproducts of the current paper is a generalization of the Godsil-McKay switching to this class of gain graphs. But our approach is definitely more general, and this is just an example of its possible applications.
 \\ \indent For general gain graphs, adjacency matrices are not complex valued but group algebra valued matrices and then  the notion of spectrum is not unique.
Anyway, complex matrices naturally arise from  the underlying graph or from the cover graphs. In \cite{JACO} a unifying approach was introduced thanks to the use of the theory of representations.
 Given a gain graph $(\Gamma,\psi)$ on a group $G$ together with a unitary representation $\pi$ of $G$, there exists a complex adjacency matrix of $(\Gamma,\psi)$ associated with $\pi$, obtained by replacing each occurrence of an element $g$ with the block  $\pi(g)$. The spectrum of the resulting matrix is called the \emph{$\pi$-spectrum} of the gain graph. With this definition, the analogous spectral characterization of balance was proved in \cite{JACO}.
But  in \cite{oncospe} further investigations on $\pi$-spectra and \emph{$\pi$-cospectrality}  highlighted the strong dependence on the choice of the representation $\pi$,
and the need for a more canonical notion of cospectrality, namely the \emph{$G$-cospectrality}, which for finite groups is equivalent to cospectrality with respect to all representations.

The main goal of the present paper is the introduction of routines, inspired by the Godsil-McKay switching \cite{GoMc} and its generalizations \cite{cos,gm}, in order to obtain pairs of $G$-cospectral gain graphs and/or pairs of $\pi$-cospectral gain graphs for some unitary representation $\pi$ of the group $G$, as well as an analysis of the relationships between the properties that make a graph suitable for the one or the other switching. \\
\indent The already known Godsil-McKay switchings can be obtained as particular cases of this more general construction.
Our construction significantly extends the scope, and we show this by providing nontrivial examples in the context of gain graphs that, if considered in the classical setting, would have been trivial or not suitable to that approach. In other words, their nontriviality is ensured precisely by the novelty of the construction.
\\ \indent The Section \ref{sec:2} contains an introduction to group representations, a summary on gain graphs and  switching isomorphisms, and a summary on the different notions of cospectrality for general gain graphs and their relations.
\\ \indent In Section \ref{sec:G} the notion of $G$-GM partition is introduced together with its associated switching; in Theorem \ref{G-Gods} it is proved that this switching produces pairs of $G$-cospectral gain graphs.
 \\ \indent In the first part of Section \ref{sec:3} the notion of $\pi$-GM partition, where $\pi$ is a unitary representation, is introduced, together with its associated switching; in Theorem \ref{lemma-Q-pi} it is proved that this switching produces pairs of $\pi$-cospectral gain graphs. Then in Theorem \ref{thm:Marcovaldo} it is proved that a partition is $G$-GM if and only if it is $\pi$-GM with respect to all unitary representations of $G$ or, equivalently, if it is $\lambda_G$-GM, where $\lambda_G$ is the left regular representation. In the second part of Section \ref{sec:3} another switching is considered: the assumptions required to be a $\pi$-GM partition are somehow weakened, as long as there exists an element of the group whose image via $\pi$ is $-I$ (when the representation is faithful, this is equivalent to the existence of a \emph{central involution}).
  \\ \indent Finally in Section \ref{sectionquaternions}, after giving a brief introduction on quaternions and quaternionic matrices, we describe a Godsil-McKay switching for  quaternion unit gain graphs and in Theorem \ref{teo:qfinale} we prove, by using the theory of the previous section,  that this switching actually produces pairs of right cospectral quaternion unit gain graphs.

\section{Preliminaries}\label{sec:2}

\subsection{Group representations and group algebras}
Let $G$ be a group with unit element $1_G$. Let $M_k(\mathbb{C})$ be the set of all square matrices of size $k$ with entries in $\mathbb C$
and let $GL_k(\mathbb{C})$ be the group of all invertible matrices in $M_k(\mathbb{C})$. A representation $\pi$ of \emph{degree $k$} (we write $\deg \pi=k$) of $G$  is a group homomorphism $\pi\colon G\to GL_k(\mathbb{C})$.\\
\indent  Let $U_k(\mathbb{C}) = \{M\in GL_k(\mathbb{C}) : M^{-1}=M^\ast\}$ be the subgroup of $GL_k(\mathbb{C})$ consisting of unitary matrices, where $M^*$ is the Hermitian transpose of $M$. The representation $\pi$ is said to be \emph{unitary} if $\pi(g)\in U_k(\mathbb{C})$, for every $g\in G$. Two representations $\pi$ and $\pi'$ of degree $k$ of $G$ are said to be equivalent, and we write $\pi\sim\pi'$, if there exists a matrix $S\in GL_k(\mathbb{C})$ such that $\pi'(g)=S^{-1}\pi(g)S$ for every $g\in G$. This is an equivalence relation and it is known that, for a finite group $G$, each equivalence class contains a unitary representative: this is the reason why working with unitary representations is not restrictive in the finite case. \\
\indent We will denote by $I_k$ (or just $I$ when the size is clear from the context) the identity matrix of size $k$ in $M_k(\mathbb{C})$.
\\\indent The \emph{kernel} of a representation $\pi$ is $\ker \pi=\{g\in G:\, \pi(g)=I\}$, and a representation is \emph{faithful} if $\ker \pi=\{1_G\}$.
\\
\indent
Given two representations $\pi_1$ and $\pi_2$ of $G$, the direct sum representation $\pi_1 \oplus \pi_2$ of $G$ is defined by
$$
(\pi_1\oplus \pi_2)(g) = \left(
                                   \begin{array}{c|c}
                                     \pi_1(g) & 0 \\ \hline
                                     0 & \pi_2(g) \\
                                   \end{array}
                                 \right) \qquad \forall \ g\in G.
$$
We will use the notation $\pi^{\oplus i}$ for the $i$-th iterated direct sum of a representation $\pi$ with itself. \\
\indent A representation $\pi$ of $G$ is said to be \emph{irreducible} if there is no nontrivial invariant subspace for the action of $G$. It is proven that $\pi$ is irreducible if and only if it is not equivalent to any direct sum of representations. It is well known that if $G$ has $m$ conjugacy classes, then there exist $m$ irreducible, pairwise  inequivalent, representations $\pi_0,\ldots,\pi_{m-1}$, which constitute a \emph{complete system of irreducible representations of $G$}, such that for each representation $\pi$ of $G$, there exist $k_0,\ldots,k_{m-1} \in \mathbb N \cup \{0\}$ such that
$\pi\sim \bigoplus_{i=0}^{m-1} \pi_i^{\oplus k_i}$.
One says that $\pi$ contains the irreducible representation $\pi_i$ with multiplicity $k_i$  if  $k_i\neq 0$, and that $\pi_i$ is a subrepresentation of $\pi$.

\indent Now we recall some remarkable representations of a group $G$. The first one is the \emph{trivial representation} $\pi_0\colon G\to \mathbb C$, such that $\pi_0(g)=1$ for every $g\in G$. It is unitary and irreducible and it  always occurs in a complete system of irreducible representations of $G$. \\
We also have the \emph{identical representation} $\pi_{id}: G \rightarrow GL_1(\mathbb{C})$, defined on each subgroup $G$ of the unit circle $\mathbb{T}=\{z\in \mathbb{C} : |z|=1\}$ as $\pi_{id}(z)=z$.\\
\indent A group $G$  naturally acts by left multiplication on itself. This action can be identified with the action of $G$ on the vector space
$\mathbb{C}G=\{\sum_{x\in G} c_x x: \ c_x\in \mathbb{C}\}$. When $G$ is finite, one gets the \emph{left regular representation} $\lambda_G$, which is faithful and has degree $|G|$. It is known that, if $\pi_0,\ldots,\pi_{m-1}$ is a complete system of irreducible representations, then $\lambda_G$ decomposes (up to equivalence) as:
$$
\lambda_G \sim \bigoplus_{i=0}^{m-1} \pi_i^{\oplus \deg \pi_i}.
$$
For further information and a more general discussion on group representation theory, we refer to \cite{fulton}.

Even when $G$ is not finite, the space of all finite $\mathbb C$-linear combinations of elements of $G$ is a vector space, denoted by $\mathbb{C}G$.
Endowed with the product
$$
\left(\sum_{x\in G} f_x x\right) \left(\sum_{y\in G} h_y y\right):= \sum_{x,y\in G} f_x h_y \, x y, \qquad \mbox{ for each } f=\sum_{x\in G} f_x x,\;h=\sum_{y\in G} h_y y\in \mathbb CG,
$$
and with the involution
$$
f^*:= \sum_{x\in G} \overline{f_{x^{-1}} }x,
$$
it is called the \emph{group algebra} of $G$.

A representation $\pi$ of degree $k$ of $G$ can be naturally extended by linearity to $\mathbb{C}G$. With a little abuse of notation, we denote with $\pi$ this extension, that is in fact a homomorphism of algebras:
\begin{equation}\label{star1}
\begin{split}
 \pi\colon \mathbb{C}G&\to M_k(\mathbb C)\\
\sum_{x\in G} f_x x&\mapsto \sum_{x\in G} f_x \pi(x).
\end{split}
\end{equation}
Moreover, if $\pi$ is unitary, then $\pi(f^*)=\pi(f)^*$ for any $f\in\mathbb CG$.

We denote by $M_{n,m}(\mathbb{C}G)$ the set of the group algebra valued $n\times m$ matrices. The product of $F\in M_{n,m}(\mathbb{C}G)$ and  $H\in M_{m,t}(\mathbb{C}G)$ is  the matrix $FH\in M_{n,t}(\mathbb{C}G)$ such that $(F H)_{i,j}=\sum_{k=1}^m F_{i,k} H_{k,j}$, where $F_{i,k} H_{k,j}$ is the product in $\mathbb C G$.
When $n=m$, we use the notation $M_n(\mathbb{C}G)$. Actually $M_n(\mathbb{C}G)$ is an algebra with  the involution $^*$ defined by $(F^*)_{i,j}=(F_{j,i})^*$, where the $^*$ on the right is the involution in $\mathbb C G$. The identity matrix $I_{M_n(\mathbb C G)}$ is defined as the diagonal matrix whose diagonal entries are equal to $1_G$, that is
$$
I_{M_n(\mathbb C G)}   = \left(
                           \begin{array}{ccc}
                             1_G &  &  \\
                              & \ddots  &  \\
                             &  & 1_G \\
                           \end{array}
                         \right) = diag(1_G,\ldots,1_G).
$$
Notice that $I_{M_n(\mathbb C G)}A = AI_{M_n(\mathbb C G)}=A$, for each $A\in M_n(\mathbb C G)$. Moreover, given $g\in G$ and $A\in M_n(\mathbb C G)$, we put
$$
gA=diag(g,g,\ldots,g)A \qquad A g = A  diag(g,g,\ldots,g).
$$
In words, the matrix $gA$ (resp. $Ag$) is obtained from $A$ by multiplying each entry on the left (resp. on the right) by the element $g\in G$. Also observe that $gA=Ag$ when $G$ is an Abelian group.

A representation $\pi: G \rightarrow GL_k(\mathbb C)$ can be further extended to $M_{n,m}(\mathbb{C}G)$. We still use the notation
\begin{eqnarray}\label{star2}
\pi\colon M_{n,m}(\mathbb{C}G)\to M_{nk,mk}(\mathbb C).
\end{eqnarray}
For $A\in M_{n,m}(\mathbb{C}G)$, the element $\pi(A)\in M_{nk,mk}(\mathbb C)$ is called the \emph{Fourier transform of $A$ at $\pi$}.
The matrix $\pi(A)$ is  obtained from $A$ by replacing each entry $a_{i,j}\in \mathbb{C}G$ with the block $\pi(a_{i,j})\in M_k(\mathbb{C})$, where $\pi$ is the map defined in Eq.\ \eqref{star1} (see  \cite[Lemma 5.1]{GLine} for the properties of this extension).

In particular this extension  $\pi\colon M_n(\mathbb{C}G)\to M_{nk}(\mathbb C)$ for square matrices  is a homomorphism of algebras (see, e.g., \cite{JACO}).
Moreover, equivalence relations and subrepresentations are also preserved in their extensions to $M_n(\mathbb{C}G)$ as the next proposition claims.

\begin{proposition}\label{productfou}{\cite[Proposition 3.4]{JACO}}
Let $A \in M_n(\mathbb C G)$ and let $\pi,\pi'$ be two representations of $G$. Then:
\begin{itemize}
\item if $\pi \sim \pi'$ then the matrices $\pi(A)$ and $\pi'(A)$ are similar;
\item the matrices $(\pi\oplus\pi')(A)$ and $\pi(A)\oplus \pi'(A)$ are similar;
\item if $\pi$ is unitary, then $\pi(A^*)=\pi(A)^*$.
\end{itemize}
\end{proposition}

Observe that the homomorphisms in Eqs. \eqref{star1} and \eqref{star2} may be not injective, even when the representation $\pi$ is faithful. To see that, let $G=\{1_G,g\}$ be the cyclic group of two elements, and consider the faithful representation $\pi: G \rightarrow GL_1(\mathbb{C})$ such that $\pi(1_G)=1$ and $\pi(g)=-1$. In particular, the nonzero element $1_G+g$ of $\mathbb{C}G$ satisfies $\pi(1_G+g) = 0$. However, this is not the case if the representation $\pi$ is (or contains) the regular representation, as the following proposition shows.

\begin{proposition}\label{prop:regolareiniettiva}
Let $G$ be a group and let $\lambda_G$ be its left regular representation. Then the map
$$
\lambda_G\colon \mathbb C G \to M_{|G|}(\mathbb C)
$$
of Eq. \eqref{star1} and the map
$$
\lambda_G\colon M_n(\mathbb{C}G)\to M_{n|G|}(\mathbb C)
$$
of Eq. \eqref{star2} are injective.
\end{proposition}
\begin{proof}
For every $g\in G$ there exist $i,j\in \{1,2,\ldots,|G|\}$ such that $gg_i=g_j$, so that $\left(\lambda_G( g)\right)_{i,j}=1$. Clearly, for all $h\in G$ with $h\neq g$, one has $hg_i\neq g_j$ and then
$\left(\lambda_G(h)\right)_{i,j}=0$.
As a consequence, for $g_1,g_2,\ldots, g_m$  distinct elements of $G$, for $c_1,c_2,\dots, c_m\in \mathbb C$, the matrix
$$
\lambda_G(c_1 g_1+\cdots + c_m g_m)=c_1\lambda_G( g_1)+\cdots+ c_m \lambda_G( g_m)
$$
has disjoint entries equals to $c_1,c_2,\ldots, c_m$. In particular if $\lambda_G(c_1 g_1+\cdots + c_m g_m)$ is the zero matrix,
 it must be $c_1=c_2=\cdots=c_m=0$.  This implies that the map $\lambda_G\colon \mathbb C G \to M_{|G|}(\mathbb C) $ is injective.

Moreover if $A\in M_n(\mathbb{C}G)$ and $A\neq 0$, then there exist $i,j \in \{1,2,\ldots,n\}$ such that $A_{i,j}\neq 0$ in $\mathbb C G $.
By virtue of the injectivity of $\lambda_G$ on $\mathbb C G$ proved above, the matrix $\lambda_G(A_{i,j} )$ has nonzero entries and so the same is true for $\lambda_G(A)$. This shows that the map $\lambda_G\colon M_n(\mathbb{C}G)\to M_{n|G|}(\mathbb C) $ is injective.
\end{proof}

\subsection{Gain graphs and switching isomorphism}
Let $\Gamma=(V_\Gamma,E_\Gamma)$ be a finite simple graph, with vertex set $V_\Gamma$ and edge set $E_\Gamma$. If $e=\{u,v\}\in E_\Gamma$   we write $u\sim v$ and we say that $u$ and $v$ are adjacent vertices of $\Gamma$.
Let $G$ be a group and consider a map $\psi$ from the set of ordered pairs of adjacent vertices of $\Gamma$ to $G$, such that $\psi(u,v)=\psi(v,u)^{-1}$. The pair $(\Gamma,\psi)$ is said to be a \emph{$G$-gain graph} (or, equivalently, a gain graph on $G$): the graph $\Gamma$ is  called  \emph{underlying graph} and $\psi$ is said a \emph{gain function} on $\Gamma$.

Suppose $|V_\Gamma|=n$ and let us fix an ordering $v_1,v_2,\ldots, v_n$ in $V_\Gamma$. The adjacency matrix of $(\Gamma,\psi)$ is the group algebra valued  matrix $A_{(\Gamma,\psi)}\in M_n(\mathbb C G)$ with entries
$$
(A_{(\Gamma,\Psi)})_{i,j}=
 \begin{cases}
 \psi(v_i,v_j) &\mbox{if } v_i\sim v_j;
  \\  0 &\mbox{otherwise.}
\end{cases}
$$
It follows from the definition that $A_{(\Gamma,\psi)}^*=A_{(\Gamma,\psi)}$.

Let $W$ be a \emph{walk} of length $h$, that is, an ordered sequence of $h+1$  (not necessarily distinct) vertices of $\Gamma$, say $w_0,w_1,\ldots, w_h$ such that $w_i\sim w_{i+1}$. One can define the gain $\psi(W)\in G$ of the walk $W$ as
$$
\psi(W):=\psi(w_0,w_1)\cdots \psi(w_{h-1},w_h).
$$
A \emph{closed walk} of length $h$ is a walk of length $h$ such that $w_0=w_{h}$.
In what follows, we denote by $\mathcal C^h(\Gamma)$ the set of all closed walks of length $h$ in $\Gamma$.
\\
\indent A gain graph $(\Gamma,\psi)$ is \emph{balanced} if $\Psi(W)=1_G$ for every closed walk $W$.
An example of a balanced gain graph is that endowed with the \emph{trivial gain function}, which is constantly equal to $1_G$.
When the underlying graph $\Gamma$ is a tree, the gain graph $(\Gamma,\psi)$ is trivially balanced.
\\A fundamental concept in the theory of gain graphs is the \emph{switching equivalence}: two gain functions $\psi_1$ and $\psi_2$ on the same underlying graph
$\Gamma$ are switching equivalent if there exists $f\colon V_\Gamma\to G$ such that
\begin{equation}\label{eqsw}
\psi_2(v_i,v_j)=f(v_i)^{-1} \psi_1(v_i,v_j)f(v_j)
\end{equation}
for any pair of adjacent vertices $v_i$ and $v_j$ of $\Gamma$. If Eq.~\eqref{eqsw} holds, we also write $\psi_2=\psi_1^f$.

It turns out that a gain graph $(\Gamma, \Psi)$ is balanced if and only it is switching equivalent to the gain graph $\Gamma$ endowed with the trivial gain function (see \cite[Lemma~5.3]{zaslavsky1} or \cite[Lemma~1.1]{refft}). Moreover, in analogy with the complex case (see \cite[Lemma~3.2]{reff1}),
the following result holds.

\begin{theorem}{\cite[Theorem 4.1]{JACO}}\label{teo-vecchio}
Let $\psi_1$ and $\psi_2$ be two gain functions on the same underlying graph $\Gamma$, with adjacency matrices $A_{(\Gamma,\psi_1)}$ and $A_{(\Gamma,\psi_2)}\in M_n(\mathbb C G)$, respectively.
Then $\psi_1$ and $\psi_2$ are switching equivalent if and only if there exists a diagonal matrix $F\in M_n(\mathbb C G)$, with $F_{i,i}\in G$ for each $i=1,\ldots, n$, such that $F^*A_{(\Gamma,\psi_1)}F=A_{(\Gamma,\psi_2)}$.
\end{theorem}
Two gain graphs $(\Gamma_1,\psi_1)$ and $(\Gamma_2,\psi_2)$ are said to be \emph{switching isomorphic} if there exists a graph isomorphism $\phi\colon V_{\Gamma_1}\to V_{\Gamma_2}$ such that $\psi_1$ is switching equivalent to the gain function $\psi_2\circ \phi$ on $\Gamma_1$ defined as $(\psi_2\circ \phi)(u,v)=\psi_2(\phi(u),\phi(v))$.

\begin{remark}\label{sw-isomorphism}
It can be seen that the adjacency matrices of switching isomorphic gain graphs can be obtained from each other by conjugating with a suitable group algebra valued matrix.
More precisely, two gain graphs $(\Gamma_1,\psi_1)$ and $(\Gamma_2,\psi_2)$ are switching isomorphic if and only if
\begin{equation*}\label{eq:swiso}
(PF)^*A_{(\Gamma_1,\psi_1)}(PF)=A_{(\Gamma_2,\psi_2)},
\end{equation*}
for some diagonal matrix $F\in M_n(\mathbb C G)$, with $F_{i,i}\in G$ for each $i=1,\ldots, n$, and for some matrix $P\in M_n(\mathbb C G)$ whose nonzero entries are equal to $1_G$ and where the  positions of the nonzero entries correspond to those of a suitable permutation matrix.
\end{remark}

\subsection{$G$-cospectrality, $\pi$-spectrum and $\pi$-cospectrality}\label{sec:4}
In this section we recall some basic definitions and properties from \cite{oncospe} about $G$-cospectrality and $\pi$-cospectrality of a gain graph $(\Gamma,\psi)$ on a group $G$, where $\pi$ is a representation of $G$.\\
\indent Let us start by defining a group algebra valued trace map $Tr\colon M_n(\mathbb C G) \to \mathbb C G$ on $M_n(\mathbb C G)$ as $Tr(A)= \sum_{i=1}^n A_{i,i}$.
Now observe that, if $A\in  M_n(\mathbb C G)$ is the adjacency matrix of a $G$-gain graph $(\Gamma,\psi)$, by virtue of \cite[Lemma~4.1]{JACO}  the entry $(A^h)_{i,j}\in\mathbb CG$ is the sum of the gains of all walks of length $h$ from $v_i$ to $v_j$.
In particular, we have:
\begin{equation*}
Tr(A^h)=\sum_{W \in\mathcal C^h(\Gamma)} \psi(W).
\end{equation*}

Let us denote by $[g]$ the conjugacy class of $g\in G$. A \emph{class function} $f\colon G\to \mathbb C$ is a map such that $g_1,g_2\in[g]\implies f(g_1)=f(g_2)$. The set  $\mathbb C_{Class}[G]$ of finitely supported class functions is a $\mathbb C$-vector space.

There exists a natural map $\mu\colon \mathbb C G\to \mathbb C_{Class}[G]$ defined as:
$$
\mu\left(\sum_{x\in G} a_x x\right)(g)=\sum_{x\in[g]} a_x.
$$
Notice that, if $G$ is Abelian, each conjugacy class contains only one element and $\mu$ is nothing but an isomorphism between $\mathbb C$-vector spaces.

\begin{definition}\label{def:cosp}
Two group algebra valued matrices $A,B\in M_n(\mathbb C G)$, with $A^*=A$ and $B^*=B$, are \emph{$G$-cospectral} if
$$
\mu(Tr(A^h))=\mu(Tr(B^h))\quad \forall h\in \mathbb N.
$$
Two gain graphs $(\Gamma_1,\psi_1)$ and $(\Gamma_2,\psi_2)$ on $G$ are \emph{$G$-cospectral} if the adjacency matrices $A_{(\Gamma_1,\psi_1)}$ and $A_{(\Gamma_2,\psi_2)}$ are $G$-cospectral.
\end{definition}
Roughly speaking, we are asking that the sequences $\{Tr(A^h)\}$ and $\{Tr(B^h)\}$ coincide, up to conjugacy. Notice that the equality of the analogous sequences for complex matrices is equivalent to the classical notion of cospectrality.

Moreover, we have shown in \cite[Theorem 4.5]{oncospe} that if two gain graphs $(\Gamma_1,\psi_1)$ and $(\Gamma_2,\psi_2)$ are switching isomorphic, then they are $G$-cospectral, whereas the inverse implication is not true. In fact, one of the main goals of the present paper consists in producing pairs of $G$-cospectral graphs which are not switching isomorphic. The following lemma holds.

\begin{lemma}\label{teo:preli}
Let $(\Gamma_1,\psi_1)$ and $(\Gamma_2,\psi_2)$ be two $G$-gain graphs and let $A,B\in M_n(\mathbb C G)$ be their adjacency matrices, respectively.
Let $C, D\in M_n(\mathbb C G)$  such that $CD=DC=I_{M_n(\mathbb C G)}$, where each entry of $D$ is supposed to be a complex multiple of $1_G$.
If $A=DB C$ then $(\Gamma_1,\psi_1)$ and $(\Gamma_2,\psi_2)$ are $G$-cospectral.
\end{lemma}
\begin{proof}
Let $S\in M_n(\mathbb C G)$. Then
$$
Tr(DS) = \sum_{i=1}^n (DS)_{i,i} = \sum_{i,k=1}^n D_{i,k}S_{k,i} = \sum_{k,i=1}^n S_{k,i}D_{i,k} = \sum_{k=1}^n (SD)_{k,k} = Tr(SD)
$$
where $D_{i,k}S_{k,i}= S_{k,i}D_{i,k}$ is true for each $i,k=1,\ldots, n$, since the entries of $D$ are all multiple of $1_G$. Now using the fact that $CD=I_{M_n(\mathbb C G)}$, we obtain $A^h = (DBC)^h = D B^h C$ for each $h$ in $\mathbb{N}$. Therefore, if we set $S = B^h C$, we get:
\begin{equation*}
Tr(A^h) = Tr(D B^h C) = Tr(B^h C D) = Tr(B^h),
\end{equation*}
and so the matrices $A$ and $B$ are $G$-cospectral, according with Definition \ref{def:cosp}.
\end{proof}

Consider a $G$-gain graph $(\Gamma,\psi)$ with adjacency matrix $A\in  M_n(\mathbb C G)$. The matrix $\pi(A)\in M_{nk}(\mathbb C)$ is called the \emph{represented adjacency matrix} of
$(\Gamma,\psi)$ with respect to the representation $\pi$.  Notice that, if $\pi$ is unitary, then by Proposition \ref{productfou} $\pi(A)$ is a Hermitian matrix and we say that its (real) spectrum is the \emph{$\pi$-spectrum of $(\Gamma,\psi)$}, denoted with $\sigma_\pi(\Gamma,\psi)$ or also $\sigma_\pi(A)$. Moreover, again by Proposition \ref{productfou}, if $\pi\sim\pi'$, then $\sigma_\pi(A)=\sigma_{\pi'}(A)$. This leads to the following definition.

\begin{definition}
Let $(\Gamma_1,\psi_1)$ and $(\Gamma_2,\psi_2)$ be two $G$-gain graphs. Let $\pi$ be a unitary representation of $G$.
The graphs $(\Gamma_1,\psi_1)$ and $(\Gamma_2,\psi_2)$ are said to be $\pi$-cospectral if $\sigma_\pi(\Gamma_1,\psi_1)=\sigma_\pi(\Gamma_2,\psi_2)$.
\end{definition}

Although the spectrum with respect to a given representation $\pi$ was formally introduced only in \cite{JACO}, particular cases of $\pi$-spectra have already been considered in the literature:
\begin{itemize}
\item If $(\Gamma,\psi)$ is a $G$-gain graph and $\pi_0$ is the trivial representation of $G$, then $\pi_{0}(A_{(\Gamma, \psi)})$ is the adjacency matrix of the underlying graph $\Gamma$, and the $\pi_0$-spectrum of $(\Gamma,\psi)$ is the spectrum of $\Gamma$.
\item If $(\Gamma,\psi)$ is a $\mathbb{T}_2$-gain graph and $\pi_{id}$ is the identical representation of $\mathbb{T}_2$, then $\pi_{id}(A_{(\Gamma, \psi)})$ is the classical adjacency matrix of the signed graph, and the $\pi_{id}$-spectrum of $(\Gamma,\psi)$ is the classical spectrum of the signed graph.
\item If $(\Gamma,\psi)$ is a $\mathbb{T}$-gain graph and $\pi_{id}$ is the identical representation of $\mathbb{T}$, then $\pi_{id}(A_{(\Gamma, \psi)})$ is the classical adjacency matrix of the complex unit gain graph, and the $\pi_{id}$-spectrum of $(\Gamma,\psi)$ is the classical spectrum of the complex unit gain graph.
\item If $(\Gamma,\psi)$ is a $G$-gain graph and $\lambda_G$ is the left regular representation of $G$, with $G$ finite, then $\lambda_G(A_{(\Gamma, \psi)})$ is the adjacency matrix of the (left) cover graph of $(\Gamma,\psi)$, and the $\lambda_G$-spectrum of $(\Gamma,\psi)$ is the spectrum of the cover graph of $(\Gamma,\psi)$.
\end{itemize}
It may happen that $(\Gamma_1,\psi_1)$ and $(\Gamma_2,\psi_2)$ are $\pi_1$-cospectral but not $\pi_2$-cospectral, even if $\pi_1$ and $\pi_2$ are faithful and irreducible \cite[Example 5.7]{oncospe}. However, if two $G$-gain graphs $(\Gamma_1,\psi_1)$ and $(\Gamma_2,\psi_2)$ are $\pi_1$-cospectral and  $\pi_2$-cospectral, with $\pi_1$ and $\pi_2$ unitary representations, then $(\Gamma_1,\psi_1)$ and $(\Gamma_2,\psi_2)$ are $(\pi_1\oplus \pi_2)$-cospectral by Proposition \ref{productfou}. The following theorem holds.

\begin{theorem}\cite[Theorem 4.14]{oncospe}\label{teo:cosp}
Let $(\Gamma_1,\psi_1)$ and $(\Gamma_2,\psi_2)$ be two $G$-gain graphs, with $G$ finite. Let $\pi_0,\ldots,\pi_{m-1}$ be a complete system of irreducible, unitary, representations of $G$. The following are equivalent.
\begin{enumerate}
\item $(\Gamma_1,\psi_1)$ and $(\Gamma_2,\psi_2)$ are $G$-cospectral;
\item $(\Gamma_1,\psi_1)$ and $(\Gamma_2,\psi_2)$ are $\pi$-cospectral, for every unitary representation $\pi$ of $G$;
\item $(\Gamma_1,\psi_1)$ and $(\Gamma_2,\psi_2)$ are $\pi_i$-cospectral, for each $i=0,\ldots,m-1$.
\end{enumerate}
\end{theorem}
It follows from Theorem \ref{teo:cosp} that, in some sense, the $G$-spectrum of a $G$-gain graph can be defined as the list of its $\pi_i$-spectra, with $\pi_i$ varying in a complete system of irreducible representations.

\section{$G$-Godsil-McKay switching}\label{sec:G}
Inspired by the classical construction given by Godsil and McKay in \cite{GoMc}, that was already generalized to signed graphs and complex unit gain graphs  \cite{cos, gm}, we present in this section a construction allowing to produce pairs of $G$-cospectral gain graphs over an arbitrary group $G$.

Let $(\Gamma,\psi)$  be a $G$-gain graph, and suppose that $\alpha=\{C_0,C_1,\ldots, C_k\}$ is a partition of the vertex set $V_\Gamma$. Let $|V_\Gamma|=n$.
For any $v\in V_\Gamma$, we set
\begin{equation}\label{eq:grado}
\Psi_i(v):=\sum_{w\in C_i,\; w\sim v} \psi(v,w)\quad\in \mathbb C G.
\end{equation}
\begin{remark}\label{remarkpi0}\rm
If $\pi_0$ is the trivial representation mapping each element of $G$ to $1$, then $\pi_0(\Psi_i(v))$ is the number of vertices  of $C_i$ which are adjacent to $v$.
\end{remark}

\begin{definition}\label{def:Gpartition}
A partition $\alpha=\left\{C_0,C_1,\ldots, C_k \right\}$ of the vertex set $V_\Gamma$ of a gain graph $(\Gamma,\psi)$  is said to be a \emph{$G$-GM partition} if:
\begin{itemize}
\item
for every $i,j\in \{1,2,\ldots, k\}$ and for every $v,v'\in C_i$, one has $$\Psi_j(v)=\Psi_j(v');$$
\item
for every $v\in C_0$, for every $i\in \{1,2,\ldots, k\}$, there exist $g_1,g_2\in G\cup \left\{0\right\}\subset \mathbb C G$ such that
$$\Psi_i(v)=\frac{|C_i|}{2}\, g_1+\frac{|C_i|}{2}\, g_2.$$
\end{itemize}
\end{definition}
In words, in a $G$-GM partition, the number of edges issuing from a vertex $v$ of $C_i$ and arriving  to vertices of $C_j$ with gain $g$, does not depend on the choice of the starting vertex in $C_i$, and this must be true for every  $i,j\in \{1,2,\ldots, k\}$  and every $g\in G$. Moreover, a vertex in the part $C_0$ must be adjacent to all vertices of a part $C_i$ (half with a gain and half with another, not necessarily distinct, gain) or to exactly half of the vertices (all with the same gain).

\begin{example}\label{exa:0}
\begin{figure}
\centering
\begin{tikzpicture}[scale=0.8]
\node[vertex]  (1) at (0,0){\tiny $v_0$};
\node[vertex]  (2) at (-2,-0.83) {\tiny $v_1$};
\node[vertex]  (3) at (-2,0.83) {\tiny $v_2$};
\node[vertex]  (4) at (-0.83,2) {\tiny $v_3$};
\node[vertex]  (5) at (0.83,2) {\tiny $v_4$};
\node[vertex]  (6) at (2,0.83){\tiny $v_5$};
\node[vertex]  (7) at (2,-0.83) {\tiny $v_6$};
\node[vertex]  (8) at (0.83,-2) {\tiny $v_7$};
\node[vertex]  (9) at (-0.83,-2) {\tiny $v_8$};
\node  (10) at (-2.2,-2) {\small $(\Gamma,\psi)$};


\draw[edge] (2) -- (3);
\draw[edge] (3) -- (4);
\draw[edge] (5) -- (4);
\draw[edge] (5) -- (6);
\draw[edge] (7) -- (6);
\draw[edge] (7) -- (8);
\draw[edge] (9) -- (8);
\draw[edge] (9) -- (2);
\draw[dedge] (1) -- (3) node[midway, above] {\tiny $i$};
\draw[dedge] (1) -- (4) node[midway, right] {\tiny $i$};
\draw[dedge] (1) -- (8) node[midway, right] {\tiny $i$};
\draw[dedge] (1) -- (9) node[midway, left] {\tiny $i$};
\draw[edge] (1) -- (2);
\draw[edge] (1) -- (7);
\draw[edge] (1) -- (5);
\draw[edge] (1) -- (6);

\node[vertex]  (1) at (0+6.3,0){ \tiny $v_0$};
\node[vertex]  (2) at (-2+6.3,-0.83) {\tiny $v_1$};
\node[vertex]  (3) at (-2+6.3,0.83) {\tiny $v_2$};
\node[vertex]  (4) at (-0.83+6.3,2) {\tiny $v_3$};
\node[vertex]  (5) at (0.83+6.3,2) {\tiny $v_4$};
\node[vertex]  (6) at (2+6.3,0.83){\tiny $v_5$};
\node[vertex]  (7) at (2+6.3,-0.83) {\tiny $v_6$};
\node[vertex]  (8) at (0.83+6.3,-2) {\tiny $v_7$};
\node[vertex]  (9) at (-0.83+6.3,-2) {\tiny $v_8$};
\node  (10) at (-2.2+6.3,-2) {\small $(\Gamma^\alpha, \psi^\alpha)$};


\draw[edge] (2) -- (3);
\draw[edge] (3) -- (4);
\draw[edge] (5) -- (4);
\draw[edge] (5) -- (6);
\draw[edge] (7) -- (6);
\draw[edge] (7) -- (8);
\draw[edge] (9) -- (8);
\draw[edge] (9) -- (2);
\draw[dedge] (1) -- (2) node[midway, above] {\tiny $i$};
\draw[dedge] (1) -- (7) node[midway, below] {\tiny $i$};
\draw[dedge] (1) -- (5) node[midway, left] {\tiny $i$};
\draw[dedge] (1) -- (6) node[midway, above] {\tiny $i$};
\draw[edge] (1) -- (3);
\draw[edge] (1) -- (4);
\draw[edge] (1) -- (8);
\draw[edge] (1) -- (9);
\end{tikzpicture}
\caption{ The $\mathbb{T}$-gain graphs $(\Gamma,\psi)$ and $(\Gamma^\alpha, \psi^\alpha)$ of Examples \ref{exa:0} and \ref{exa:01}.}\label{fig:0}
\end{figure}
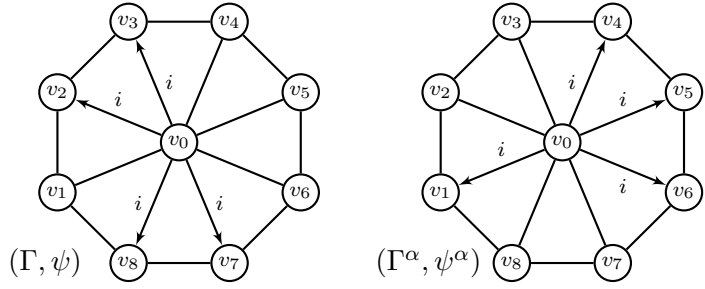

Let $(\Gamma,\psi) $ be the $\mathbb{T}$-gain graph represented in Fig.~\ref{fig:0}, where the undirected solid lines correspond to edges labeled by $1$, whereas the orientation from $v_0$ to $v_j$, for $j\in\{2,3,7,8\}$ is due to the nontrivial gains $\psi(v_0, v_j)  = i$ (and therefore, $\psi(v_j, v_0)= i^{-1}=-i)$. If we put $C_0 = \{v_0\}$ and $C_1 = \{v_j : j= 1, \ldots, 8\}$, then it is easy to check that $\alpha=\{C_0, C_1\}$ is a $\mathbb{T}$-GM partition.  In fact
$$
\Psi_1(v_0)=4\cdot 1+4\cdot i \quad \mbox{ and }\quad  \Psi_1(v_j)=2 \quad  \mbox{ for } j= 1, \ldots, 8.
$$
\end{example}

\begin{definition}\label{def:gm-s}
Let $\alpha$ be a  \emph{$G$-GM partition} of the gain graph $(\Gamma,\psi)$. The graph $(\Gamma,\psi)^\alpha=(\Gamma^\alpha,\psi^\alpha)$ is the gain graph such that:
\begin{itemize}
\item the adjacency and the gains between pairs of vertices in $\bigcup_{i=1}^k C_i$ are the same as in $(\Gamma,\psi)$;
\item for $v\in C_0$ and $i\in \{1,2,\ldots, k\}$ such that  $\Psi_i(v)=\frac{|C_i|}{2}\, g_1+\frac{|C_i|}{2}\, g_2$,  with $g_1,g_2\in G\cup \left\{0\right\}$, and for every $w\in C_i$, we set:
$$\psi^\alpha(v,w)=
\begin{cases}
g_1 &\mbox{ if } \psi(v,w)=g_2,\\
g_2 &\mbox{ if } \psi(v,w)=g_1,
\end{cases}
$$
with the convention that $\psi(v,w)=0$ (resp. $\psi^\alpha(v,w)=0$) means that $v$ and $w$ are not adjacent in $\Gamma$ (resp. in $\Gamma^\alpha$).
\end{itemize}
A $G$-GM partition $\alpha$ is said to be \emph{nontrivial} if $(\Gamma,\psi)$ and $(\Gamma,\psi)^\alpha$ are not switching isomorphic.
\end{definition}
\noindent For each $i=0,1,\ldots, k$, put $|C_i|=n_i$.
With a given $G$-GM partition $\alpha=\left\{C_0,C_1,\ldots, C_k \right\}$, we associate a group algebra valued matrix $ Q_\alpha\in M_n(\mathbb C G) $ in the following way:
$$
\left(Q_\alpha\right)_{u,v}=
\begin{cases}
1_G &\mbox{ if } u=v\in C_0\\
 \frac{2-n_i}{n_i}\,1_G &\mbox{ if } u=v\in C_i,\; i=1,\ldots, k\\
\frac{2}{n_i}\,1_G &\mbox{ if } u,v\in C_i,\;u \neq v,\; i=1,\ldots, k \\
0 &\mbox{ otherwise.}
\end{cases}
$$
Observe that the nonzero entries of $Q_\alpha$ are all real multiples of $1_G$ and $Q_\alpha^\ast = Q_\alpha$. Let $
J_{M_n(\mathbb{C}G)}\in M_n(\mathbb{C}G)$  be the square matrix of size $n$ whose entries are all equal to $1_G\in\mathbb CG$.
Notice that, if we put
\begin{eqnarray}\label{Qn}
Q_{n} = \frac{2}{n}J_{M_{n}(\mathbb{C}G)} - I_{M_{n}(\mathbb{C}G)},
\end{eqnarray}
then it holds $Q_{n}^\ast = Q_{n}$ and we have:
$$
Q_\alpha = \left(
  \begin{array}{cccc}
1_{M_{n_0}(\mathbb{C}G)} & 0 & \ldots & 0 \\
0 & Q_{n_1} & \ldots & 0 \\
\vdots & \vdots & \ddots & \vdots\\
0 & 0 & \ldots & Q_{n_k}
  \end{array}
\right).
$$
Observe that the matrix $Q_n$ is inspired by the matrix considered in \cite[Theorem 2.2]{GoMc}.

\begin{lemma}\label{propQ}
Let $Q_n$ be the matrix defined in Eq. \eqref{Qn}. Then the following formulas hold.
\begin{enumerate}
\item $Q_{n}^2 = 1_{M_{n}(\mathbb{C}G)}$.
\item Let $X \in M_{m,n}(\mathbb{C}G)$ with constant row sums $S_r$ and constant column sums $S_c$. Then $Q_{m}XQ_{n} = X$.
\item Let $\lambda \in \mathbb{C}G$ and let ${\bf x}= (\lambda, \ldots, \lambda)^T$ be a constant column of length $n$. Then $Q_{n}{\bf x} = {\bf x}$.
\item Let ${\bf x} = (x_1, \ldots , x_{n})^T$ be a column of length $n$, with $x_i \in \mathbb{C}G$ for each $i=1,\ldots, n$, such that $\sum_{i = 1}^{n} x_i = 0$. Then
    $Q_{n} {\bf x} = -{\bf x}$.
\item Let $g_1, g_2$ in $G\cup \{0\}$ and let ${\bf x}$ be a column of length $n$, such that half entries of ${\bf x}$ are equal to $g_1$ and half entries are equal to $g_2$. Then $Q_{n} {\bf x} = (g_1+g_2)\cdot (1_G, \ldots , 1_G)^T - {\bf x}$.
\end{enumerate}
\end{lemma}
\begin{proof}
The proof is analogous to those of Lemma 2.1 and Proposition 2.3 of \cite{gm}. To give an idea of how computations work when dealing with elements from $\mathbb{C}G$ instead of complex numbers, we focus here only on the proof of property (2). Observe that the matrix $X$ can be rewritten as
$$
X = g_1X_1+g_2X_2 + \cdots + g_k X_k,
$$
where $g_i\in G$ and the entries of $X_i$ belong to the set $\{0,1_G\}$, for each $i=1,\ldots, k$. Moreover, by assumption, each matrix $X_i$ has constant row sums and constant column sums. Since the matrices $X_i, Q_m, Q_n$ have entries which are complex multiple of $1_G$, they behave as complex matrices in all computations. Therefore, it follows from Lemma 2.1 of \cite{gm} that $Q_m X_i Q_n = X_i$ for each $i=1,\ldots, k$. By linearity and using the fact that $g_i1_G=1_Gg_i$ for each $i$, the property (2) follows.
\end{proof}
About property (5), observe that the vector $Q_{n} {\bf x}$ is obtained from ${\bf x}$ by switching the entries $g_1$ and $g_2$. In particular, the property (5) recovers the property (3) when $g_1=g_2=\lambda$.

\begin{theorem}\label{G-Gods}
Let $(\Gamma,\psi)$ be a $G$-gain graph and let $\alpha$ be a $G$-GM partition, with associated matrix $ Q_\alpha \in M_n( \mathbb C G)$. Then
\begin{equation}\label{eq:fonda}
 A_{(\Gamma^\alpha,\psi^\alpha)}= Q_\alpha A_{(\Gamma,\psi)} Q_\alpha.
 \end{equation}
 In particular, $(\Gamma,\psi)$ and $(\Gamma,\psi)^\alpha$ are $G$-cospectral
\end{theorem}
\begin{proof}
We can choose a suitable labeling of $V_\Gamma$ such that
$$
A_{(\Gamma, \psi)} = \left( \begin{array}{ccccc}
C_{0,0} & C_{0,1} & \ldots & C_{0,k-1} & C_{0,k} \\
    C_{0,1}^* & C_{1,1} & \ldots & C_{1, k-1} & C_{1,k} \\
    \vdots & \vdots & \ddots &  \vdots & \vdots \\
    C_{0,k-1}^* & C_{1, k-1}^* & \ldots & C_{k-1,k-1} & C_{k-1,k} \\
    C_{0,k}^* & C_{1, k}^* & \ldots & C_{k-1,k}^* & C_{k,k}
    \end{array}   \right)
$$
where, for all $i,j\in \{0,1,\ldots,k\}$, the block $C_{i,j}\in M_{n_i,n_j}(\mathbb{C}G)$ describes the adjacencies from vertices in $C_i$ to vertices in $C_j$, so that:
\begin{itemize}
\item the matrix $C_{0,0}$ satisfies $C_{0,0}^\ast = C_{0,0}$;
\item for each $i = 1, \ldots, k$, the matrix $C_{i,i}$ satisfies $C_{i,i}^\ast = C_{i,i}$, and it has constant row sum and constant column sum;
\item for each $i,j = 1, \ldots, k$ the matrix $C_{i,j}$ has constant row sum and constant column sum;
\item for each $i = 1, \ldots, k$, the rows of the matrix $C_{0,i}$ (or, equivalently, the columns of the matrix $C_{0,i}^\ast$) satisfy the property (5) of Lemma \ref{propQ}.
\end{itemize}
By virtue of Lemma \ref{propQ}, we obtain:    \footnotesize
\begin{equation*}
\begin{split}
    Q_\alpha A_{(\Gamma, \psi)}Q_\alpha &= \left( \begin{array}{ccccc} C_{0,0} & C_{0,1}Q_{n_1} & \ldots & C_{0,k-1}Q_{n_{k-1}} &C_{0,k}  Q_{n_k} \\
    Q_{n_1} C_{0,1}^* & Q_{n_1}C_{1,1}Q_{n_1} & \ldots & Q_{n_1}C_{1, k-1}Q_{n_{k-1}} & Q_{n_1}C_{1,k}Q_{n_k} \\
    \vdots & \vdots & \ddots &  \vdots & \vdots \\
    Q_{n_{k-1}} C_{0,k-1}^* & Q_{n_{k-1}}C_{1, k-1}^*Q_{n_1} & \ldots & Q_{n_{k-1}}C_{k-1,k-1}Q_{n_{k-1}} & Q_{n_{k-1}}C_{k-1,k}Q_{n_{k}} \\
    Q_{n_k} C_{0,k}^* & Q_{n_{k}} C_{1, k}^*Q_{n_1} & \ldots & Q_{n_{k}}C_{k-1,k}^*Q_{n_{k-1}} & Q_{n_{k}}C_{k,k}Q_{n_{k}}
    \end{array}\right) \\
        &= \left( \begin{array}{ccccc} C_{0,0} & C_{0,1}Q_{n_1} & \ldots & C_{0,k-1}Q_{n_{k-1}} & C_{0,k}Q_{n_k} \\
    Q_{n_1} C_{0,1}^* & C_{1,1} & \ldots & C_{1, k-1} & C_{1,k} \\
    \vdots & \vdots & \ddots &  \vdots & \vdots \\
    Q_{n_{k-1}} C_{0,k-1}^* & C_{1, k-1}^* & \ldots & C_{k-1,k-1} & C_{k-1,k} \\
    Q_{n_k} C_{0,k}^* & C_{1, k}^* & \ldots & C_{k-1,k}^* & C_{k,k}
   \end{array} \right)
\end{split}
\end{equation*}   \normalsize
where we have used the property (2) of Lemma \ref{propQ} to say that, for each $i,j=1,\ldots, k$, one has $Q_{n_i}C_{i, j}Q_{n_j}=C_{i,j}$. Finally, by using property (5) of Lemma \ref{propQ} and the definition of $(\Gamma, \psi)^\alpha$, we can conclude that $Q_\alpha A_{(\Gamma, \psi)}Q_\alpha = A_{(\Gamma^\alpha, \psi^\alpha)}$, since the matrix $C_{0,i}Q_{n_{i}}$ (together with its Hermitian transpose $Q_{n_{i}}C_{0,i}^\ast$) corresponds to the adjacencies between the part $C_0$ and the part $C_i$ of $(\Gamma, \psi)^\alpha$, for each $i=1,\ldots, k$. This completes the proof of Eq.~\eqref{eq:fonda}.
The last statement  is a direct consequence of Lemma \ref{teo:preli}, with $C=D=Q_\alpha$.
\end{proof}

Observe that, if $\alpha$ is a $G$-GM partition of $V_\Gamma$, then by Remark \ref{remarkpi0} $\alpha$ is also a good partition in the sense of the classical Godsil-McKay switching \cite{GoMc}. In particular, by applying the trivial representation $\pi_0$ to the members of Eq. \eqref{eq:fonda}, one recovers the cospectrality of the two graphs obtained with the classical construction. In some sense, the cospectrality produced by the $G$-Godsil-McKay construction implies the cospectrality of the underlying graphs according to the classical construction. Moreover, this remark suggests the possibility of investigating Eq. \eqref{eq:fonda} with respect to any other representation $\pi$ of $G$, by using the adjacency matrix with respect to that representation. This is exactly what we are going to do in Section \ref{sec:3}.

The next example shows that the presence of gains makes the situation very rich and complex: in fact, it may happen that the $G$-Godsil-McKay construction produces gain graphs whose underlying graphs are isomorphic, even if the corresponding $G$-cospectral gain graphs are not switching isomorphic, when regarded as gain graphs.

\begin{example}\label{exa:01}
In Example \ref{exa:0} we showed that $\alpha=\{C_0, C_1\}$ is a $G$-GM partition for the $\mathbb{T}$-gain graph  $(\Gamma,\psi) $ represented in  Fig.~\ref{fig:0}.
In Fig.~\ref{fig:0} also the $\mathbb{T}$-gain graph  $(\Gamma^\alpha, \psi^\alpha)$ is depicted, obtained according to Definition \ref{def:gm-s}.
\\ \indent By  Theorem \ref{G-Gods} the gain graphs $(\Gamma, \psi)$ and $(\Gamma^\alpha, \psi^\alpha)$ are $\mathbb{T}$-cospectral. Moreover, their underlying graphs $\Gamma$ and $\Gamma^\alpha$ are  isomorphic. However, $(\Gamma, \psi)$ and $(\Gamma^\alpha, \psi^\alpha)$ are not switching isomorphic as gain graphs, and so $\alpha$ is nontrivial.
\\ \indent We will prove it by contradiction. Let us assume that there exist a switching function $\xi\colon V_\Gamma \to \mathbb{T}$ and a graph automorphism $\phi\colon V_\Gamma \to V_{\Gamma^\alpha}$ such that
\begin{equation}\label{eq:ultima}
\xi(v_i)^{-1}\psi(v_i, v_{j})\xi(v_{j}) = \psi^\alpha(\phi(v_i), \phi(v_{j})) \quad \mbox{for any}\quad  v_i\sim v_j.
\end{equation}
Clearly it must be $\phi(v_0)=v_0$ and  both $\psi$ and $ \psi^\alpha$ take only value $1$ on pair of adjacent vertices in $C_1$. As a consequence of Eq. \eqref{eq:ultima} we have
$$1=\xi(v_1)^{-1}\xi(v_{2})=\xi(v_2)^{-1}\xi(v_{3})=\cdots=\xi(v_8)^{-1}\xi(v_{1}),$$
that is, $\xi$ must be constant on $C_1$. Put $t:= \xi(v_1)$. By considering Eq. \eqref{eq:ultima} for the pair of adjacent vertices $v_0,v_1$, we have
$$\xi(v_0)^{-1}t=\xi(v_0)^{-1}\psi(v_0, v_1)\xi(v_{1}) = \psi^\alpha(\phi(v_0), \phi(v_1))= \psi^\alpha(v_0, \phi(v_1))\in \{1,i\}.$$

 \indent Let us focus on the first case: if $\xi(v_0) = t$, then $\xi$ is a constant function and so we obtain that $\psi = \psi^\alpha \circ \phi$. Observe that an automorphism of $\Gamma$ acts on the set $\{v_1,\ldots, v_8\}$ as an element of the dihedral group $D_{16}$, which is the group of symmetries of an octagon. On the other hand, in order to preserve the gains, such a permutation should send $\{v_1, v_4, v_5, v_6\}$ to $\{v_2, v_3, v_7, v_8\}$. However, it is not difficult to see that no element of $D_{16}$ satisfies this property. This is a contradiction.\\
\indent Let us study the second case: $\xi(v_0)^{-1} t = i$. This implies that $t = i \xi(v_0)$ and so, by considering Eq. \eqref{eq:ultima} for the pair of adjacent vertices $v_0,v_2$, we obtain:
\begin{equation*}
\psi^\alpha(\phi(v_0), \phi(v_2)) = \xi(v_0)^{-1} \psi(v_0, v_2) t =  i^2 =  -1.
\end{equation*}
This is impossible, since $-1$ does not appear as a gain in $(\Gamma^\alpha,\psi^\alpha)$.
\end{example}

\section{Godsil-McKay switching with respect to a representation}\label{sec:3}
$G$-cospectrality is a very strong property, introduced in \cite{oncospe}, which implies cospectrality with respect to all representations. In the previous literature, only the notion of cospectrality for gain graphs with respect to some privileged representation of the gain group has been considered, rather than $G$-cospectrality.
For example, for signed graphs and complex unit gain graphs, the classical notion of spectrum and cospectrality is equivalent to that of $\pi_{id}$-cospectrality (see Section \ref{sec:4}), and two cospectral gain graphs  are not necessarily G-cospectral. As we will see in the next section, also for quaternion unit gain graphs the spectrum so far considered in the literature is that associated with a particular representation.        \\
\indent For this reason, it is convenient to have a more flexible routine working also for gain graphs with no $G$-GM partition, producing gain graphs that are cospectral at least with respect to some representations.       \\
\indent We will see two constructions: the first in Section \ref{sub1}, which is valid for any group $G$ and any unitary representation $\pi$; the second in Section \ref{sub2}, which is applicable only to those representations whose range contains $-I$ (which in the case of faithful representations implies the existence of a central involution in $G$). This second procedure generalizes Godsil-McKay switching for signed graphs and complex unit graphs \cite{cos,gm}.

\subsection{$\pi$-GM switching}\label{sub1}
Before describing this switching it is useful to make an observation on what happens when one considers representations that are not faithful.
\begin{remark}\label{rem:fedele}
Let $(\Gamma_1,\psi_1)$ and $(\Gamma_2,\psi_2)$ be two $G$-gain graphs and let $\pi$ be a unitary representation of $G$. Suppose that, for any pair of adjacent vertices $u$ and $v$, one has
$$
\psi_1(u,v)\psi_2(u,v)^{-1}\in \ker \pi.
$$
Then $(\Gamma_1,\psi_1)$ and $(\Gamma_2,\psi_2)$ are $\pi$-cospectral. In other words, if we modify $(\Gamma,\psi)$ by multiplying (on the left and/or on the right) some gains by elements in $\ker \pi$ (consistently for both orientations, in such a way that we still have a gain function), we do not change the $\pi$-spectrum: this is obvious because we actually do not change the represented adjacency matrix $\pi( A_{(\Gamma,\psi)})$.
\end{remark}

It is important to keep this observation in mind, because the switching  we are going to describe is not defined up to the operation of multiplying by elements of the kernel. On the contrary, a preliminary multiplication of this kind could even transform a gain graph not suitable for the switching  into one that is suitable.
\begin{definition}\label{def:pipar}
A partition $\alpha=\left\{C_0,C_1,\ldots, C_k \right\}$ of the vertex set $V_\Gamma$ of a gain graph $(\Gamma,\psi)$  is said to be a \emph{$\pi$-GM partition} if:
\begin{itemize}
\item
for every $i,j\in \{1,2,\ldots, k\}$ and for every $v,v'\in C_i$, one has
$$
\pi(\Psi_j(v))=\pi(\Psi_j(v'));
$$
\item for every $v\in C_0$, for every $i\in \{1,2,\ldots, k\}$, there exist $g_1,g_2\in G\cup \left\{0\right\}\subset \mathbb C G$ such that
$$
\Psi_i(v)=\frac{|C_i|}{2}\, g_1+\frac{|C_i|}{2}\, g_2.
$$
\end{itemize}
\end{definition}

The analogy with the Definition \ref{def:Gpartition} of a $G$-GM partition is clear and in fact a $G$-GM partition is a $\pi$-GM partition, for every representation $\pi$ of $G$.
But a $\pi$-GM partition is not necessarily a $G$-GM partition, as Example \ref{exa:pi1} shows. Moreover, a partition $\alpha$ is a $\pi_0$-GM partition, where $\pi_0$ is the trivial representation, if and only if it is a suitable partition for the classic Godsil-McKay switching of the underlying graph $\Gamma$.

\begin{example}\label{exa:pi1}
Let  $(\Gamma,\psi)$ be the $S_4$-gain graph depicted in Fig.~\ref{fig:1}, where $S_4$ is the symmetric group on $4$ elements. The usual rules apply: an edge without label represents an edge with trivial gain $1_{S_4}$, and when the gain is an involution the orientation of the associated edge is not specified.
Let $\pi_p \colon S_4\to U_4(\mathbb C)$ be the \emph{permutation representation}, associating with each permutation its canonical $\{0,1\}$-valued permutation matrix.
Let $\alpha=\{C_0,C_1\}$ be the partition with $C_0=\{v_1\}$ and $C_1=V_\Gamma\setminus C_0$.
At first observe that $v_1$ is adjacent to exactly half of the vertices of $C_1$, all with the same gain $1_{S_4}$. Then the second condition of Definition \ref{def:pipar} holds. Moreover:
\begin{align*}
&\Psi_1(v_2)=\Psi_1(v_3)=\Psi_1(v_4)=\Psi_1(v_5)=1_{S_4}+(1 2)(3 4)\\
& \Psi_1(v_6)=\Psi_1(v_7)=\Psi_1(v_8)=\Psi_1(v_9)=(1 2)+(3 4).
\end{align*}
This implies that $\alpha$ is not a $S_4$-GM partition. But
$$
\pi_p(1_{S_4}+(1 2)(3 4))= \begin{pmatrix}
1&1&0&0\\
1&1&0&0\\
0&0&1&1\\
0&0&1&1\\
\end{pmatrix}=\pi_p((1 2)+(3 4))
$$
and then $\alpha$ is $\pi_p$-GM. Let $\pi_s$ be the sign representation associating $sgn(\sigma)\in\{\pm1\}$ with any $\sigma\in S_4$ according to its parity: then the partition $\alpha$ is not a $\pi_s$-GM partition, since:
$$
\pi_s(\Psi_1(v_2))=\pi_s(1_{S_4}+(1 2)(3 4))=2,\qquad \pi_s(\Psi_1(v_6))= \pi_s((1 2)+(3 4))=-2.
$$
Finally, if $\pi_0$ denotes the trivial representation, the partition $\alpha$ is a $\pi_0$-GM partition, since the (underlying) graph induced by $C_1$ is regular and the unique vertex of $C_0$ is adjacent to exactly half of the vertices of $C_1$.

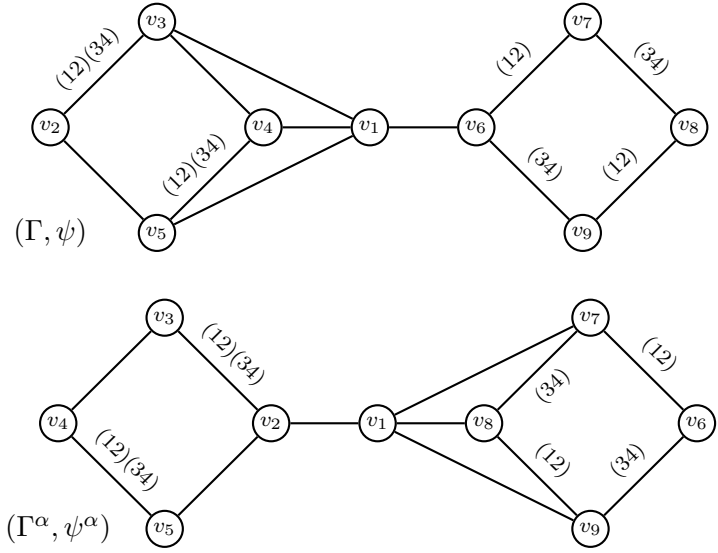
\begin{figure}
\centering
\begin{tikzpicture}[scale=0.7]
\node[vertex]  (1) at (0,0){\tiny $v_2$};
\node[vertex]  (2) at (2,2) {\tiny $v_3$};
\node[vertex]  (3) at (4,0) {\tiny $v_4$};
\node[vertex]  (4) at (2,-2) {\tiny $v_5$};
\node[vertex]  (5) at (6,0) {\tiny $v_1$};
\node[vertex]  (6) at (8,0){\tiny $v_6$};
\node[vertex]  (7) at (10,2) {\tiny $v_7$};
\node[vertex]  (8) at (12,0) {\tiny $v_8$};
\node[vertex]  (9) at (10,-2) {\tiny $v_9$};
\node  (finto) at (10,-3) {};
\node  (10) at (0,-2) {\small $(\Gamma,\psi)$};

\draw[edge] (1) -- (2) node[midway, above,sloped] {\tiny $(12)(34)$};
\draw[edge] (3) -- (4) node[midway, above, sloped] {\tiny $(12)(34)$};
\draw[edge] (3) -- (2);
\draw[edge] (1) -- (4);
\draw[edge] (2) -- (5);
\draw[edge] (3) -- (5);
\draw[edge] (4) -- (5);
\draw[edge] (6) -- (5);
\draw[edge] (6) -- (7) node[midway, above,sloped] {\tiny $(12)$};
\draw[edge] (8) -- (9) node[midway, above, sloped] {\tiny $(12)$};
\draw[edge] (8) -- (7) node[midway, above,sloped] {\tiny $(34)$};
\draw[edge] (6) -- (9) node[midway, above, sloped] {\tiny $(34)$};
\end{tikzpicture}
\centering
\begin{tikzpicture}[scale=0.7]
\node[vertex]  (1) at (4,0){\tiny $v_2$};
\node[vertex]  (2) at (2,2) {\tiny $v_3$};
\node[vertex]  (3) at (0,0) {\tiny $v_4$};
\node[vertex]  (4) at (2,-2) {\tiny $v_5$};
\node[vertex]  (5) at (6,0) {\tiny $v_1$};
\node[vertex]  (6) at (12,0){\tiny $v_6$};
\node[vertex]  (7) at (10,2) {\tiny $v_7$};
\node[vertex]  (8) at (8,0) {\tiny $v_8$};
\node[vertex]  (9) at (10,-2) {\tiny $v_9$};
\node  (10) at (0,-2) {\small  $(\Gamma^\alpha,\psi^\alpha)$};

\draw[edge] (1) -- (2) node[midway, above,sloped] {\tiny $(12)(34)$};
\draw[edge] (3) -- (4) node[midway, above, sloped] {\tiny $(12)(34)$};
\draw[edge] (3) -- (2);
\draw[edge] (1) -- (4);
\draw[edge] (1) -- (5);
\draw[edge] (7) -- (5);
\draw[edge] (8) -- (5);
\draw[edge] (9) -- (5);
\draw[edge] (6) -- (7) node[midway, above,sloped] {\tiny $(12)$};
\draw[edge] (8) -- (9) node[midway, above, sloped] {\tiny $(12)$};
\draw[edge] (8) -- (7) node[midway, below,sloped] {\tiny $(34)$};
\draw[edge] (6) -- (9) node[midway, above, sloped] {\tiny $(34)$};
\end{tikzpicture}
\caption{The $S_4$-gain graphs $(\Gamma,\psi)$ and $(\Gamma^\alpha,\psi^\alpha)$ of Examples \ref{exa:pi1} and \ref{exa:pi2}.}\label{fig:1}
\end{figure}
\end{example}

Even if $\alpha$ is a $\pi$-GM partition of a gain graph $(\Gamma,\psi)$, and not necessarily a $G$-GM partition, we can still construct   the matrix $Q_\alpha$ and the gain graph $(\Gamma^\alpha,\psi^\alpha)$ exactly as in Definition \ref{def:gm-s}, and the following theorem holds.

\begin{theorem}\label{lemma-Q-pi}
Let $(\Gamma,\psi)$ be a $G$-gain graph, let $\pi$ be a unitary representation of $G$, and let $\alpha$ be a $\pi$-GM partition.Then
$$
\pi(A_{(\Gamma^\alpha,\psi^\alpha)})=  \pi(Q_\alpha)  \pi(A_{(\Gamma,\psi)})  \pi(Q_\alpha).
$$
In particular, $(\Gamma,\psi)$ and $(\Gamma,\psi)^\alpha$ are $\pi$-cospectral.
\end{theorem}
\begin{proof}
The extension of $\pi$ to $M_n(\mathbb C G)$ is a homomorphism, then one has  $\pi(Q_\alpha A_{(\Gamma,\psi)} Q_\alpha) = \pi(Q_\alpha)  \pi(A_{(\Gamma,\psi)})  \pi(Q_\alpha)$.
By definition of $Q_\alpha$ we have:      \footnotesize
$$
Q_\alpha A_{(\Gamma, \psi)}Q_\alpha =
\left( \begin{array}{ccccc} C_{0,0} & C_{0,1}Q_{n_1} & \ldots & C_{0,k-1}Q_{n_{k-1}} &C_{0,k}  Q_{n_k} \\
    Q_{n_1} C_{0,1}^* & Q_{n_1}C_{1,1}Q_{n_1} & \ldots & Q_{n_1}C_{1, k-1}Q_{n_{k-1}} & Q_{n_1}C_{1,k}Q_{n_k} \\
    \vdots & \vdots & \ddots &  \vdots & \vdots \\
    Q_{n_{k-1}} C_{0,k-1}^* & Q_{n_{k-1}}C_{1, k-1}^*Q_{n_1} & \ldots & Q_{n_{k-1}}C_{k-1,k-1}Q_{n_{k-1}} & Q_{n_{k-1}}C_{k-1,k}Q_{n_{k}} \\
    Q_{n_k} C_{0,k}^* & Q_{n_{k}} C_{1, k}^*Q_{n_1} & \ldots & Q_{n_{k}}C_{k-1,k}^*Q_{n_{k-1}} & Q_{n_{k}}C_{k,k}Q_{n_{k}}
    \end{array}\right).
$$      \normalsize
Notice that, as in the proof of Theorem \ref{G-Gods}, by virtue of property (5) of Lemma \ref{propQ},
the matrix $C_{0,i}Q_{n_{i}}$ (together with its Hermitian transpose $Q_{n_{i}}C_{0,i}^\ast$) corresponds to the adjacencies between the part $C_0$ and the part $C_i$ of $(\Gamma, \psi)^\alpha$, for each $i=1,\ldots, k$.
By construction of $(\Gamma^\alpha,\psi^\alpha)$ the adjacencies and the gains of the parts $C_1,\ldots, C_k$ are the same as in  $(\Gamma, \psi)$ and then  in order to get the statement we only need to prove that
$$
\pi(Q_{n_i} C_{i,j} Q_{n_j}) = \pi(C_{i,j}),\qquad \forall i,j\in\{1,2,\ldots,k\}.
$$
Let us set $C_i=\{v_1,\ldots v_{n_i} \}$ and $C_j=\{w_1,\ldots w_{n_j}\}$ in such a way that the element in position $h,l$ of the block $C_{i,j}$ is $\psi(v_h, w_l)$ if $v_h$ and $w_l$ are adjacent, and $0$ otherwise.\\
Notice that the sum of the elements of the $h$-th row of $C_{i,j}$ is $\Psi_j(v_h)$, and the sum along the $l$-th column is $\Psi_i(w_l)^\ast$.\\
Differently from what happens for a $G$-GM partition, the row sums and the column sums are not constant as elements of $\mathbb C G$, but the first  property of Definition \ref{def:pipar} ensures that the possible differences all live in the kernel of the extension of $\pi$ to $\mathbb C G$. Therefore $C_{i,j}$ can be written as the sum of a constant row and column sum matrix with a matrix in the kernel of the extension of $\pi$ to $M_{n_i,n_j}( \mathbb C G)$.
Let us denote by $S\in \mathbb C G$ the sum of all entries of $C_{i,j}$. Let us define $D\in M_{n_i,n_j}(\mathbb{C}G)$ by:
\begin{equation*}
\begin{split}
D_{h,l}&=
\begin{cases}
 \frac{S}{n_i}-\Psi_j(v_h) &\mbox{ if } h\neq 1 \mbox{ and } l= 1\\
 \frac{S}{n_j}-\Psi_i(w_l)^\ast &\mbox{ if } h=1 \mbox{ and } l\neq 1\\
 \frac{S}{n_i}+ \frac{S}{n_j}-\Psi_j(v_1)-\Psi_i(w_1)^\ast &\mbox{ if } h= 1 \mbox{ and } l= 1\\
0 &\mbox{ otherwise.}\\
\end{cases}\\&=
\left( \begin{array}{cccc}
 \frac{S}{n_i}+ \frac{S}{n_j}-\Psi_j(v_1)-\Psi_i(w_1)^\ast&  \frac{S}{n_j}-\Psi_i(w_2)^\ast &\cdots & \frac{S}{n_j}-\Psi_i(w_{n_j})^\ast
\\
 \frac{S}{n_i}-\Psi_j(v_2) &0 &\cdots & 0
\\
\vdots & \vdots & \ddots & \vdots\\
 \frac{S}{n_i}-\Psi_j(v_{n_i}) &0 &\cdots & 0
    \end{array}\right).
\end{split}
\end{equation*}
It is easy to check that actually $\pi(D)=0$, since  by the first property of Definition \ref{def:pipar} one has
$$
\pi\left(\frac{S}{n_i}\right)=\pi\left(\Psi_j(v_h)\right),\qquad \pi\left(\frac{S}{n_j}\right)=\pi\left(\Psi_i(w_l)^\ast\right),\quad  \forall h\in\{1,\ldots, n_i\},l\in\{1,\ldots, n_j\}.
$$
Let us prove that the matrix $C_{i,j}+D$ has constant row sums and constant column sums in $\mathbb CG$.
Clearly for $h>1$ and $l>1$, the sum of the elements of the $h$-th row  of $C_{i,j}+D$ is  $\frac{S}{n_i}$ and the sum of the elements of the $l$-th column of $C_{i,j}+D$ is  $\frac{S}{n_j}$. Let us check that the same is true for the first row and the first column:
\begin{equation*}
\begin{split}
\sum_{l=1}^{n_j} \left( C_{i,j}+D   \right)_{1,l}&= \sum_{l=1}^{n_j} \left( C_{i,j}\right)_{1,l} + \frac{S}{n_i}+ \frac{S}{n_j}-\Psi_j(v_1)-\Psi_i(w_1)^\ast + \sum_{l=2}^{n_j}\left(  \frac{S}{n_j}-\Psi_i(w_l)^\ast\right)\\
&=\Psi_j(v_1) + \frac{S}{n_i}+ \frac{S}{n_j}-\Psi_j(v_1)-\Psi_i(w_1)^\ast + \frac{n_j-1}{n_j} S- \sum_{l=2}^{n_j} \Psi_i(w_l)^\ast\\
&=  \frac{S}{n_i}+ S-S=\frac{S}{n_i},
\end{split}
\end{equation*}
and similarly
\begin{equation*}
\begin{split}
\sum_{h=1}^{n_i} \left( C_{i,j}+D   \right)_{h,1}&= \sum_{h=1}^{n_i} \left( C_{i,j}\right)_{h,1} + \frac{S}{n_i}+ \frac{S}{n_j}-\Psi_j(v_1)-\Psi_i(w_1)^\ast + \sum_{h=2}^{n_i}\left(  \frac{S}{n_i}-\Psi_j(v_h)\right)\\
&=\Psi_i(w_1)^\ast + \frac{S}{n_i}+ \frac{S}{n_j}-\Psi_j(v_1)-\Psi_i(w_1)^\ast + \frac{n_i-1}{n_i} S- \sum_{h=2}^{n_i} \Psi_j(v_h)\\
&=  \frac{S}{n_j}+ S-S=\frac{S}{n_j}.
\end{split}
\end{equation*}
Now using property (2) of Lemma \ref{propQ} and the property of the extension of $\pi$ (see \cite[Lemma 5.1]{GLine}) we can conclude the proof:
\begin{equation*}
\begin{split}
\pi(   Q_{n_i} C_{i,j} Q_{n_j}) &= \pi(   Q_{n_i} \left(C_{i,j}+D- D\right) Q_{n_j})= \pi(   Q_{n_i} \left(C_{i,j}+D\right) Q_{n_j}-   Q_{n_i}D Q_{n_j})
\\&=  \pi(   C_{i,j}+D-   Q_{n_i}D Q_{n_j})=\pi(C_{i,j}) +0 - \pi(Q_{n_i})\,0\,\pi(Q_{n_j})=\pi(C_{i,j}).
\end{split}
\end{equation*}
\end{proof}

\begin{example}\label{exa:pi2}
In Example \ref{exa:pi1} we presented an $S_4$-gain graph $(\Gamma,\psi)$ and a $\pi_p$-GM partition $\alpha$ of its vertex set. In Fig.~\ref{fig:1} the associated gain graph $(\Gamma^\alpha,\psi^\alpha)$
  is depicted. By an explicit computation one can check that  the $36\times 36$ complex matrices $\pi_p( A_{(\Gamma,\psi)})$ and $\pi_p( A_{(\Gamma^\alpha,\psi^\alpha)})$ are cospectral and so $(\Gamma,\psi)$ and $(\Gamma^\alpha,\psi^\alpha)$ are $\pi_p$-cospectral. However, it is possible to see that $\pi_s( A_{(\Gamma,\psi)})$ and $\pi_s( A_{(\Gamma^\alpha,\psi^\alpha)})$ are not cospectral and so
 $(\Gamma,\psi)$ and $(\Gamma^\alpha,\psi^\alpha)$ are not $\pi_s$-cospectral. Finally, notice that the underlying graphs $\Gamma$ and $\Gamma^\alpha$ are cospectral (actually, they are isomorphic),
 as we expected since the partition $\alpha$ is also a $\pi_0$-GM partition.
\end{example}

\begin{remark}\label{rem:ker2}
Notice that for a partition $\alpha=\{C_0,C_1,\ldots,C_k\}$ the second condition of Definition \ref{def:pipar} is the same as in Definition \ref{def:Gpartition}: for any $v\in C_0$ and any $i\in\{1,\ldots,k\}$, if there are edges from $v$ to $C_i$ then
half of the vertices of $C_i$ are adjacent to $v$ with a gain $g_1\in G$ and the other half are adjacent to $v$ with a gain $g_2\in G$ or are not adjacent at all.
Actually, we may weaken this condition by asking that half of the vertices have gain with the same image via $\pi$: this would make sense if $\pi$ were not faithful.  Anyway elements in $G$ with the same image via $\pi$
are equal up to multiplication by elements of $\ker \pi$, this  means that this generalization would produce nothing that could not be achieved by combining  the switching with the operation described in Remark \ref{rem:fedele} of multiplication of some gains by some  elements of $\ker \pi$.
\end{remark}

\begin{example}\label{exa:pi3}
In Fig.~\ref{fig:3} four $S_4$-gain graphs are depicted. We consider  the partition $\alpha=\{C_0,C_1\}$ on $\Gamma$, with $C_0=\{v_7\}$ and $C_1=V_\Gamma\setminus \{v_7\}$, and the sign representation $\pi_s$ of $S_4$. The partition $\alpha$ is not $\pi_s$-GM for  $(\Gamma,\psi_1)$, because the second condition of Definition \ref{def:pipar} is not satisfied.
In the light of Remark \ref{rem:ker2} one can notice that $v_7$ is adjacent to half of the vertices of $C_1$ with gains having  the same image equal to $1$ via $\pi_s$.
Then one can consider the gain graph
 $(\Gamma,\psi_2)$ that is $\pi_s$-cospectral with $(\Gamma,\psi_1)$ by Remark  \ref{rem:fedele}, since it is obtained from $(\Gamma,\psi_1)$ by multiplying one edge by an element in $\ker \pi_s$.
 Now it is easy to check that $\alpha$ is $\pi_s$-GM for  $(\Gamma,\psi_2)$ and then $(\Gamma,\psi_2)$ and $(\Gamma^\alpha,\psi_2^\alpha)$ are $\pi_s$-cospectral. Finally we can again multiply some edges by elements in $\ker \pi_s$ obtaining $(\Gamma^\alpha,\psi_3)$, that by Remark  \ref{rem:fedele} is $\pi_s$-cospectral with $(\Gamma^\alpha,\psi_2^\alpha)$. Finally, by transitivity, the gain graphs $(\Gamma,\psi_1)$ and $(\Gamma^\alpha,\psi_3)$ are $\pi_s$-cospectral.
\begin{figure}
\centering
\begin{tikzpicture}[scale=1.5]


\node[vertex]  (1a) at (-2.7,0){\tiny $v_1$};

\node[vertex]  (2a) at (-2.7,2){\tiny $v_2$};
\node[vertex]  (4a) at (2-2.7,0){\tiny $v_4$};

\node[vertex]  (3a) at (2-2.7,2){\tiny $v_3$};

\node[vertex]  (5a) at (1-2.7,1.5){\tiny $v_5$};

\node[vertex]  (6a) at (1.5-2.7,1){\tiny $v_6$};

\node[vertex]  (7a) at (1-2.7,1){\tiny $v_7$};


\draw[dedge] (4a) -- (3a) node[midway, sloped, above,rotate=180] {\tiny $(123)$};

\draw[edge] (1a) -- (2a);

\draw[edge] (2a) -- (3a)node[midway, above,sloped] {\tiny $(12)$};

\draw[edge] (4a) -- (1a)node[midway, below,sloped] {\tiny $(13)$};

\draw[edge] (1a) -- (7a)node[midway, above,sloped] {\tiny $(12)(34)$};
\draw[edge] (7a) -- (5a);
\draw[edge] (7a) -- (6a);


\node[vertex]  (1) at (0,0){\tiny $v_1$};

\node[vertex]  (2) at (0,2){\tiny $v_2$};

\node[vertex]  (4) at (2,0){\tiny $v_4$};

\node[vertex]  (3) at (2,2){\tiny $v_3$};

\node[vertex]  (5) at (1,1.5){\tiny $v_5$};

\node[vertex]  (6) at (1.5,1){\tiny $v_6$};

\node[vertex]  (7) at (1,1){\tiny $v_7$};


\draw[edge] (1) -- (2);

\draw[edge] (2) -- (3)node[midway, above,sloped] {\tiny $(12)$};

\draw[dedge] (4) -- (3)node[midway, above,sloped] {\tiny $(123)$};

\draw[edge] (4) -- (1)node[midway, below,sloped] {\tiny $(13)$};

\draw[edge] (1) -- (7);
\draw[edge] (7) -- (5);
\draw[edge] (7) -- (6);


\node[vertex]  (1) at (0+2.7,0){\tiny $v_1$};

\node[vertex]  (2) at (0+2.7,2){\tiny $v_2$};

\node[vertex]  (4) at (2+2.7,0){\tiny $v_4$};

\node[vertex]  (3) at (2+2.7,2){\tiny $v_3$};

\node[vertex]  (5) at (1+2.7,1.5){\tiny $v_5$};

\node[vertex]  (6) at (1.5+2.7,1){\tiny $v_6$};

\node[vertex]  (7) at (1+2.7,1){\tiny $v_7$};


\draw[edge] (1) -- (2);

\draw[edge] (2) -- (3)node[midway, above,sloped] {\tiny $(12)$};

\draw[dedge] (4) -- (3)node[midway, above,sloped] {\tiny $(123)$};

\draw[edge] (4) -- (1)node[midway, below,sloped] {\tiny $(13)$};

\draw[edge] (7) -- (2);
\draw[edge] (7) -- (3);
\draw[edge] (7) -- (4);


\node[vertex]  (1) at (0+5.4,0){\tiny $v_1$};

\node[vertex]  (2) at (0+5.4,2){\tiny $v_2$};

\node[vertex]  (4) at (2+5.4,0){\tiny $v_4$};

\node[vertex]  (3) at (2+5.4,2){\tiny $v_3$};

\node[vertex]  (5) at (1+5.4,1.5){\tiny $v_5$};

\node[vertex]  (6) at (1.5+5.4,1){\tiny $v_6$};

\node[vertex]  (7) at (1+5.4,1){\tiny $v_7$};


\draw[edge] (1) -- (2);

\draw[edge] (2) -- (3)node[midway, above,sloped] {\tiny $(12)$};

\draw[dedge] (4) -- (3)node[midway, above,sloped] {\tiny $(143)$};

\draw[edge] (4) -- (1)node[midway, below,sloped] {\tiny $(13)$};

\draw[edge] (7) -- (2)node[midway, below,sloped] {\tiny $(12)(34)$};
\draw[edge] (7) -- (3);
\draw[edge] (7) -- (4)node[midway, below,sloped] {\tiny $(12)(34)$};
\end{tikzpicture}
\caption{From left to right, the $S_4$-gain graphs $(\Gamma,\psi_1)$, $(\Gamma,\psi_2)$, $(\Gamma^\alpha,\psi_2^\alpha)$ and $(\Gamma^\alpha,\psi_3)$
of Example \ref{exa:pi3}.}\label{fig:3}
\end{figure}
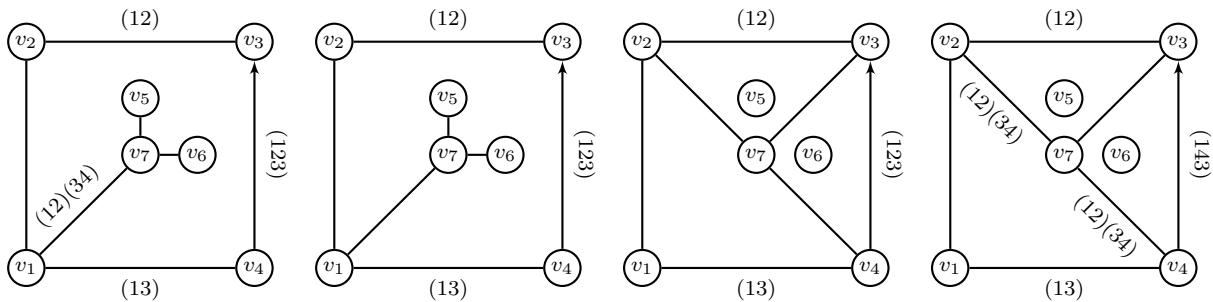
\end{example}

If two gain graphs are cospectral with respect to some unitary representations, then they are cospectral with respect to any direct sum of those representations, while the converse is, in general, false  (\cite[Example~5.4]{oncospe}). On the other hand, for the cospectrality obtained via the by GM-switching, both the implications hold.
\begin{proposition}\label{prop:somma}
Let $(\Gamma,\psi)$ be a $G$-gain graph and let $\pi_1,\pi_2$ be unitary representations of $G$. A partition $\alpha$ is a $\pi_1$-GM partition and a $\pi_2$-GM partition for $(\Gamma,\psi)$ if and only if  it is
a $(\pi_1\oplus \pi_2)$-GM partition for $(\Gamma,\psi)$.
\end{proposition}
\begin{proof}
The second property of Definition \ref{def:pipar} does not depend on the representation. Therefore, in order to prove the claim, it is enough to show that, for a partition $\alpha=\left\{C_0,C_1,\ldots, C_k \right\}$, the property
$$
 \pi_1(\Psi_j(v))=\pi_1(\Psi_j(v')) \mbox{ and } \pi_2(\Psi_j(v))=\pi_2(\Psi_j(v')),\quad \forall i,j\in \{1,2,\ldots, k\}, \forall v,v'\in C_i
$$
is equivalent to
$$
(\pi_1\oplus \pi_2) (\Psi_j(v))=(\pi_1\oplus \pi_2) (\Psi_j(v')),\quad \forall i,j\in \{1,2,\ldots, k\}, \forall v,v'\in C_i.$$
This is true since merely at a matrix level one has
$$
A_1=A_2 \mbox{ and }  B_1=B_2  \iff A_1\oplus B_1=A_2\oplus B_2.
$$
\end{proof}
In the light of Proposition \ref{prop:somma} it is not reductive to restrict, if necessary, to irreducible representations.

We have already noticed that a $G$-GM partition is $\pi$-GM for every unitary representation. At least when the group $G$ is finite, we know from Theorem \ref{teo:cosp} that
$G$-cospectrality is equivalent to $\pi$-cospectrality for all (irreducible)  representations of $G$.

The next theorem implies that for the cospectrality obtained by the GM-switchings the analogous result holds; moreover, it is enough to look at the regular representation $\lambda_G$ (which is not true for the cospectrality in general instead  \cite[Remark~4.15]{oncospe}).

\begin{theorem} \label{thm:Marcovaldo}
Let $G$ be a finite group and let $\pi_0,\ldots,\pi_{m-1}$ be a complete system of irreducible representations of $G$. Let
$(\Gamma,\psi)$ be a $G$-gain graph and let $\alpha$ be a partition of the vertex set $V_\Gamma$. The following are equivalent.
\begin{enumerate}
\item $\alpha$ is a $G$-GM partition;
\item  $\alpha$ is a $\pi$-GM partition for every representation $\pi$;
\item  $\alpha$ is a $\pi_i$-GM partition for each $i=0,\ldots, m-1$;
\item  $\alpha$ is a $\lambda_G$-GM partition.
\end{enumerate}
\end{theorem}
\begin{proof}
Implications (1)$\implies$ (2)$\implies$ (3) are obvious; implication (3)$\implies$ (4) follows from Proposition \ref{prop:somma}.\\
Let us prove implication (4)$\implies$ (1).
Suppose that $\alpha=\{C_0,\ldots, C_k\}$ is a $\lambda_G$-GM partition. The second condition of Definition \ref{def:pipar} is equal to that of Definition  \ref{def:Gpartition}, then we only need to prove that
$$
\Psi_j(v)=\Psi_j(v'),\quad \forall i,j\in \{1,2,\ldots, k\}, \forall v,v'\in C_i
$$
by assuming that
$$
\lambda_G(\Psi_j(v))=\lambda_G(\Psi_j(v')),\quad \forall i,j\in \{1,2,\ldots, k\}, \forall v,v'\in C_i.
$$
By Proposition \ref{prop:regolareiniettiva} the extension of  $\lambda_G$ to $\mathbb C G $ is injective and the claim follows.
\end{proof}

It follows that $G$-GM switching and $\lambda_G$-GM switching are the same. It is worth noting  that this is not true for any faithful representation: we have already seen in the Example \ref{exa:pi1} that a partition may be a $\pi$-GM partition, with $\pi$ faithful, without being a $G$-GM partition.

\subsection{$\pi$-GM switching when $G$ has a central involution}\label{sub2}
In this section we consider a gain group $G$ and a representation $\pi$ for which there exists an element $s\in G$ such that $\pi(s)=-I$. Notice that if the representation $\pi$ is faithful, the element $s$ must be a central involution of $G$. When such an element exists, we can further weaken the condition for a partition to be a $\pi$-GM partition. Therefore we can update our definition.
\begin{definition}\label{def:updated}
A partition $\alpha=\left\{C_0,C_1,\ldots, C_k \right\}$ of the vertex set $V_\Gamma$ of a gain graph $(\Gamma,\psi)$  is said to be a \emph{$\pi$-GM partition} if:
\begin{itemize}
\item for every  $i,j\in \{1,2,\ldots, k\}$ and for every $v,v'\in C_i$, one has $$\pi(\Psi_j(v))=\pi(\Psi_j(v'));$$
\item for every $v\in C_0$, for every $i\in \{1,2,\ldots, k\}$, at least one of the following conditions holds:
\begin{enumerate}
\item there exist $g_1,g_2\in G\cup \left\{0\right\}\subset \mathbb C G$ such that
$\Psi_i(v)=\frac{|C_i|}{2}\, g_1+\frac{|C_i|}{2}\, g_2,$
\item there exists $s\in G$ such that $\pi(s)=-I$ and $\pi(\Psi_i(v))=0$.
\end{enumerate}
\end{itemize}
\end{definition}

\begin{example}\label{exa:diedrale1}
Let
$$
D_8=\langle a, b| a^4=b^2=1_{D_8},  bab=a^{-1} \rangle
$$
be the \emph{Dihedral group} of order $8$. Recall that the elements of $D_8$ can be regarded as the symmetries of a square: in particular, $a$ is a rotation of $\pi/2$ and $b$ is a reflection. Let us denote by $\pi_2$ the representation of $D_8$ of degree $2$ such that
$$
\pi_2(a)= \begin{pmatrix}
0&-1\\
1&0\\\end{pmatrix},\qquad\pi_2(b)= \begin{pmatrix}
1&0\\
0&-1\\\end{pmatrix}.
$$
Observe that $\pi_2$ is a faithful representation, and that the element $a^2$ is a central involution such that $\pi_2(a^2) = -I$.
On the left of Fig.~\ref{fig:diedrale1} the $D_8$-gain graph $(\Gamma,\psi)$ is depicted. Let us consider the partition $\alpha=\{C_0,C_1,C_2\}$ with $C_0=\{v_7,v_8\}$, $C_1=\{v_1,v_2, v_3,v_4\}$ and $C_2=\{v_5,v_6\}.$ \\ One has
$$\Psi_1(v_i)=
\begin{cases}
2\cdot1_{D_8}  &\mbox{ if } i\in \{1,2,3,4,8\}\\
 a+a^3 &\mbox{ if } i\in \{5,6,7\}\\
\end{cases}\qquad
\Psi_2(v_i)=
\begin{cases}
0  &\mbox{ if } i\in \{1,3,5,6,7,8\}\\
 a+a^3 &\mbox{ if } i\in \{2,4\}.\\
\end{cases}
$$
The partition $\alpha$ is not a $\pi_2$-GM partition with respect to the definition given in Section \ref{sub1}: one can check that the only obstruction is that
 the second condition of Definition \ref{def:pipar} does not hold for $v_7$. But we have
$$
\pi_2(\Psi_1(v_7))=\begin{pmatrix}
0&-1\\
1&0\\\end{pmatrix}+\begin{pmatrix}
0&-1\\
1&0\\\end{pmatrix}^3=\begin{pmatrix}
0&0\\
0&0\\\end{pmatrix},
$$
and then $\alpha$ is a $\pi_2$-GM partition according to Definition \ref{def:updated}.

\begin{figure}
\centering

\begin{tikzpicture}[scale=1.4]


\node[vertex]  (1) at (0,0){\tiny $v_1$};

\node[vertex]  (2) at (0,2){\tiny $v_2$};

\node[vertex]  (3) at (2,2){\tiny $v_3$};

\node[vertex]  (4) at (2,0){\tiny $v_4$};

\node[vertex]  (5) at (3.5,2){\tiny $v_7$};

\node[vertex]  (6) at (2,3.5){\tiny $v_8$};

\node[vertex]  (7) at (1.25,1.25){\tiny $v_5$};

\node[vertex]  (8) at (0.75, 0.75){\tiny $v_{6}$};

\node  (10) at (0.5,3.3) {\small $(\Gamma,\psi)$};


\draw[edge] (1) -- (2);
\draw[edge] (2) -- (3);
\draw[edge] (4) -- (3);
\draw[edge] (1) -- (4);

\draw[dedge] (5) --(3)node[midway, below ,sloped] {\tiny $a$};
\draw[dedge] (4) -- (5)node[midway, below] {\tiny $a$};

\draw[dedge] (2) --(7)node[midway, above ] {\tiny $a$};
\draw[dedge] (8) -- (2)node[midway, below left] {\tiny $a$};
\draw[dedge] (4) --(8)node[midway, below ] {\tiny $a$};
\draw[dedge] (7) -- (4)node[midway, above right] {\tiny $a$};

\draw[edge] (6) -- (3);
\draw[edge] (6) -- (2);
\draw[edge] (6) -- (5)node[midway, right] {\tiny $b$};

\node[vertex]  (1) at (0+5,1+0.5){\tiny $v_1$};

\node[vertex]  (2) at (0+5,3+0.5){\tiny $v_2$};

\node[vertex]  (3) at (2+5,3+0.5){\tiny $v_3$};

\node[vertex]  (4) at (2+5,1+0.5){\tiny $v_4$};

\node[vertex]  (5) at (3.5+5,1+0.5){\tiny $v_7$};

\node[vertex]  (6) at (2+5,0){\tiny $v_8$};

\node[vertex]  (7) at (1.25+5,1.25+1+0.5){\tiny $v_5$};

\node[vertex]  (8) at (0.75+5,0.75+1+0.5){\tiny $v_{6}$};

\node  (10) at (5.5,0.4) {\small $(\Gamma,\psi)^{\alpha,a^2}$};

\draw[edge] (1) -- (2);
\draw[edge] (2) -- (3);
\draw[edge] (4) -- (3);
\draw[edge] (1) -- (4);

\draw[dedge] (3) --(5)node[midway, above right] {\tiny $a$};
\draw[dedge] (5) -- (4)node[midway, above] {\tiny $a$};

\draw[dedge] (2) --(7)node[midway, above ] {\tiny $a$};
\draw[dedge] (8) -- (2)node[midway, below left] {\tiny $a$};
\draw[dedge] (4) --(8)node[midway, below ] {\tiny $a$};
\draw[dedge] (7) -- (4)node[midway, above right] {\tiny $a$};

\draw[edge] (6) -- (4);
\draw[edge] (6) -- (1);
\draw[edge] (6) -- (5)node[midway, right] {\tiny $b$};
\end{tikzpicture}
\caption{The $D_8$-gain graphs $(\Gamma,\psi)$ and $(\Gamma,\psi)^{\alpha,a^2}$ of Examples \ref{exa:diedrale1}, \ref{exa:diedrale2} and \ref{exa:diedrale3}.}\label{fig:diedrale1}
\end{figure}
\end{example}

\begin{remark}
Proposition \ref{prop:somma} is still valid with this updated definition.
The first property of Definition \ref{def:updated} is in fact exactly the same of Definition \ref{def:pipar}. If the second property in Definition \ref{def:updated} holds for $\pi_1\oplus \pi_2$ then the same is true for both $\pi_1$ and $\pi_2$. Vice versa,  for a vertex $v\in C_0$ and for $i\in \{1,2,\ldots, k\}$,  condition (1)  of the second property in Definition \ref{def:updated} holds or not independently of the representation: the only thing to prove is that if  condition (2)  of the second property in Definition \ref{def:updated} holds for both $\pi_1$ and $\pi_2$, then it holds for $\pi_1\oplus \pi_2$. But this is simply the consequence of how the direct sum of representations and matrices is defined.
Therefore also Theorem \ref{thm:Marcovaldo} can be proved, with this updated definition, with an analogous proof, just taking into account that for the left regular representation $\lambda_G$ the condition (2) of the second property of  Definition \ref{def:updated} never occurs.
\end{remark}

This time we need to describe the switched graph, because it may be different from the one described for $G$-GM partitions and actually it can depend on the chosen representation.
\begin{definition}\label{def:consommanulla}
Let $\alpha$ be a \emph{$\pi$-GM partition} of $(\Gamma,\psi)$ and let $s\in G$ be such that $\pi(s)=-I$. The gain graph $(\Gamma,\psi)^{\alpha,s}=(\Gamma^{\alpha,s},\psi^{\alpha,s})$ is defined as follows:
\begin{itemize}
\item the adjacency and the gains between pairs of vertices in $\bigcup_{i=1}^k C_i$ are as in $(\Gamma,\psi)$;
\item for $v\in C_0$ and $i\in \{1,2,\ldots, k\}$, if we are in the case $(1)$ that $\Psi_i(v)=\frac{|C_i|}{2}\, g_1+\frac{|C_i|}{2}\, g_2$,  with $g_1,g_2\in G\cup \left\{0\right\}$, for every $w\in C_i$, then we set:
$$
\psi^{\alpha,s}(v,w)=
\begin{cases}
g_1 &\mbox{ if } \psi(v,w)=g_2,\\
g_2 &\mbox{ if } \psi(v,w)=g_1,
\end{cases}
$$
with the convention that $\psi(v,w)=0$ (resp. $\psi^{\alpha,s}(v,w)=0$) means that $v$ and $w$ are not adjacent in $\Gamma$ (resp. in $\Gamma^{\alpha,s}$).
If we are not in the case $(1)$ and we are in the case  $(2)$, so that $\pi(\Psi_i(v))=0$, then we set
$$
\psi^{\alpha,s}(v,w)=s\; \psi(v,w).
$$
\end{itemize}
The partition $\alpha$ is said to be \emph{nontrivial} if   $(\Gamma,\psi)$ and  $(\Gamma,\psi)^{\alpha,s}$ are not switching isomorphic.
\end{definition}

\begin{example}\label{exa:diedrale2}
On the right of Fig.~\ref{fig:diedrale1} the $D_8$-gain graph $(\Gamma,\psi)^{\alpha,a^2}$ is depicted, associated to the  $\pi_2$-GM partition $\alpha$ of $(\Gamma,\psi)$ discussed in Example \ref{exa:diedrale1}.
In fact in $D_8$ the element $a^2$ is a central involution and $\pi_2(a^2)=-I$.\\
Notice that $\Gamma$ and $\Gamma^{\alpha,a^2}$ are not isomorphic, since the vertex $v_4$ of $\Gamma^{\alpha,a^2}$ has degree $6$, but no vertex in $\Gamma$ has such a degree. In particular, $(\Gamma,\psi)$ and $(\Gamma,\psi)^{\alpha,a^2}$ are not switching isomorphic and then $\alpha$ is nontrivial.
It is worth noting that the underlying graph $\Gamma^{\alpha,a^2}$  is not the graph that would have been obtained with a classic Godsil-McKay switching on $\Gamma$ (which in fact respects all the conditions described in \cite{GoMc}).
\end{example}

Observe that the multiplication by $s$ is the analogue of the switch of the sign in the Godsil-McKay switching defined for signed and complex unit gain graphs; more precisely,  for subgroups of $\mathbb T$ containing $s=-1$ and with respect to the representation $\pi_{id}$, the switching described in  Definition \ref{def:consommanulla} is  the same of that  defined in \cite{cos,gm}.
Notice that if both (1)  and  (2) holds,  we have arbitrarily decided that (1) has priority over (2) in the construction of $(\Gamma,\psi)^{\alpha,s}$. This is not completely irrelevant, because when the representation $\pi$ is not faithful, the switch of $g_1$ with $g_2$ described for case $(1)$ may not produce the same gain graph obtained by multiplying by $s$ instead.

\begin{theorem}\label{lemma-Q-pi-sommanulla}
Let $(\Gamma,\psi)$ be a $G$-gain graph, let $\pi$ be a unitary representation of $G$ and $\alpha$ be a $\pi$-GM partition. Then
$$
\pi(A_{(\Gamma^{\alpha,s},\psi^{\alpha,s})})=  \pi(Q_\alpha)  \pi(A_{(\Gamma,\psi)})  \pi(Q_\alpha).
$$
In particular, $(\Gamma,\psi)$ and $(\Gamma,\psi)^{\alpha,s}$ are $\pi$-cospectral.
\end{theorem}
\begin{proof}
A direct computation gives:       \footnotesize
$$
Q_\alpha A_{(\Gamma, \psi)}Q_\alpha = \left( \begin{array}{ccccc} C_{0,0} & C_{0,1}Q_{n_1} & \ldots & C_{0,k-1}Q_{n_{k-1}} &C_{0,k}  Q_{n_k} \\
    Q_{n_1} C_{0,1}^* & Q_{n_1}C_{1,1}Q_{n_1} & \ldots & Q_{n_1}C_{1, k-1}Q_{n_{k-1}} & Q_{n_1}C_{1,k}Q_{n_k} \\
    \vdots & \vdots & \ddots &  \vdots & \vdots \\
    Q_{n_{k-1}} C_{0,k-1}^* & Q_{n_{k-1}}C_{1, k-1}^*Q_{n_1} & \ldots & Q_{n_{k-1}}C_{k-1,k-1}Q_{n_{k-1}} & Q_{n_{k-1}}C_{k-1,k}Q_{n_{k}} \\
    Q_{n_k} C_{0,k}^* & Q_{n_{k}} C_{1, k}^*Q_{n_1} & \ldots & Q_{n_{k}}C_{k-1,k}^*Q_{n_{k-1}} & Q_{n_{k}}C_{k,k}Q_{n_{k}}
    \end{array}\right).
$$     \normalsize
As in the proof of Theorem \ref{lemma-Q-pi}
we know that $$\pi(Q_{n_i}C_{i, j}Q_{n_{j}})=\pi(C_{i,j})$$ for $i,j\in\{1,\ldots,k\}$, and then the blocks in position  $i,j$  of $\pi(Q_\alpha)  \pi(A_{(\Gamma,\psi)})  \pi(Q_\alpha)$ and $ \pi(A_{(\Gamma^{\alpha,s},\psi^{\alpha,s})})$ coincide. We have to prove that for any $i\in\{1,\ldots,k\}$ the matrix
$ Q_{n_i} C_{0,i}^*$ and the associated block in $A_{(\Gamma^{\alpha,s},\psi^{\alpha,s})}$ have the same image via $\pi$.
As in the proof of Theorem \ref{G-Gods}, the columns of $ Q_{n_i} C_{0,i}^*$
associated with vertices $v$ of $C_0$  such that there exist $g_1,g_2\in G\cup\{0\}$ with
$\Psi_i(v)=\frac{|C_i|}{2}\, g_1+\frac{|C_i|}{2}\, g_2$, coincide with the corresponding columns in $A_{(\Gamma^{\alpha,s},\psi^{\alpha,s})}$. In particular, they have the same $\pi$-image.
Suppose now that $v\in C_0$ is such that
$$
\pi(\Psi_i(v))=0,
$$
and that there exists $s\in G$ such that $\pi(s)=-I$. Let us denote by $\bold{x}$ the associated column in $C_{0,i}^*$. More precisely, if $C_i=\{v_1,\ldots, v_{n_i}\}$, the entry $\bold{x}_l$ is $\psi(v,v_l)^*$ if $v$ and $v_l$ are adjacent and $0$ otherwise.
We define a constant column $\bold{y}$ whose image via $\pi$ is the zero column as:
$$
\bold{y}_l=\frac{\Psi_i(v)^*}{n_i}, \qquad l\in \{1,\ldots,n_i\}.
$$
In fact, one has $\pi(\Psi_i(v)^\ast)=0$ since $\pi(\Psi_i(v))=0$. Also notice that
$$
\sum_{l=1}^{n_i} (\bold{x} - \bold{y})_l=0,
$$
and so, by properties (3) and (4) of Lemma \ref{propQ} one has:
$$
Q_{n_i} \bold{x}=Q_{n_i} (\bold{x}-\bold{y})+Q_{n_i}\bold{y}=2\bold{y}-\bold{x}
$$
and then
$$
\pi(Q_{n_i} \bold{x})=-\pi(\bold{x}),
$$
since $\pi(\Psi_i(v))=0$. On the other hand, by definition of $(\Gamma^{\alpha,s},\psi^{\alpha,s})$, the column corresponding to the adjacency between the vertex $v$ with the part $C_i$ is exactly $s \, \bold{x}$, whose image via $\pi$ is $-\pi(\bold{x})$, and this concludes the proof.
\end{proof}

\begin{example}\label{exa:diedrale3}
An explicit computation shows that the characteristic polynomial of the $\pi_2$-adjacency matrix of both the $D_8$-gain graphs $(\Gamma,\psi)$ and $(\Gamma,\psi)^{\alpha,a^2}$ of Example \ref{exa:diedrale2} (depicted in Fig.~\ref{fig:diedrale1}) is
$$
x^{16}-26x^{14}+263x^{12}-1306x^{10}+3297x^8-3968x^6+1984x^4-256x^2.
$$
\end{example}

\section{Godsil-McKay switching for quaternion unit gain graphs}\label{sectionquaternions}
Let $\mathbb{H}$ be the algebra of real quaternions, that is, the unital associative $\mathbb{R}$-algebra whose generators are $i,j$ and $k$, where
\begin{equation*}
i^2 = j^2 = k^2 = ijk = -1.
\end{equation*}
Every element $q\in \mathbb{H}$ can be written as
\begin{eqnarray}\label{expressionq}
q = a + bi + cj + dk,
\end{eqnarray}
where the coefficients $a,b,c,d\in \mathbb{R}$ are uniquely determined. Given $q\in \mathbb{H}$ as in Eq. \eqref{expressionq}, the \emph{imaginary part} of $q$ is $Im(q) = bi + cj + dk$, whereas the \emph{real part} of $q$ is $Re(q) = a$.
The \emph{conjugate} of a quaternion $q$ is defined as $\overline{q} := Re(q) - Im(q)$. Notice that $\mathbb{H}$ is not a commutative algebra.\\ Given $q\in \mathbb{H}$ as in Eq. \eqref{expressionq}, the \emph{norm} of $q$ is defined as
\begin{equation*}
|q| = \sqrt{a^2 + b^2 + c^2 + d^2}.
\end{equation*}
Notice that $|q_1 q_2| = |q_1||q_2|$. The \emph{inverse} of an element $q\neq 0\in \mathbb{H}$ is then defined as
$$
q^{-1} = \frac{\overline{q}}{|q|^2}
$$
and it satisfies the equalities $qq^{-1} = q^{-1}q = 1$. Two quaternions $q_1, q_2 \in \mathbb{H}$ are said to be \emph{similar} if there exists $h\in \mathbb{H}$ such that
\begin{equation*}
q_1 = h^{-1} q_2 h.
\end{equation*}
If $q_1$ and $q_2$ are similar, then we write $q_1 \sim q_2$. The similarity is clearly an equivalence relation, and we denote by $[q]$ the equivalence class of $q$ with respect to this equivalence. Observe that if $Im(q) = 0$, that is $q$ is real, then $[q] = \{q\}$, since in this case $qh = hq$ for each $h$ in $\mathbb{H}$. Moreover, one has $\overline{q}\sim q$ for each $q\in \mathbb{H}$.\\
\indent By \cite[Lemma 2.2]{Spec} we know that for each $q\in \mathbb{H}$ there is a unique $\lambda_q\in \mathbb{C}$ such that $Im(\lambda_q) \geq 0$ and
\begin{equation*}
[q] = [\lambda_q] = [\overline{\lambda_q}].
\end{equation*}
In words, we can say that there exists a 1-to-1 correspondence between the set of similarity classes of quaternions and the set of complex numbers with nonnegative imaginary part. Finally, we recall that the set
$$
U(\mathbb{H}) = \{q \in \mathbb{H} : |q|=1\}
$$
of quaternions whose norm is equal to $1$ constitutes a multiplicative group. The main object of this section is the investigation of $U(\mathbb{H})$-gain graphs.

Let us put $\mathbb{H}^n = \{(q_1,\ldots, q_n): q_i\in \mathbb{H} \ \forall i=1,\ldots, n\}$ and let us denote by $M_n(\mathbb{H})$ the set of $n\times n$ matrices with entries in $\mathbb{H}$. As in the complex case, given a matrix $A =(a_{ij})\in M_n(\mathbb{H})$, its \emph{conjugate transpose} is the matrix $A^*= (\overline{a_{ji}}) \in M_n(\mathbb{H})$. A matrix $A \in M_n(\mathbb{H})$ is \emph{Hermitian} if $A= A^*$, it is \emph{normal} if $AA^* = A^*A$ and it is \emph{unitary} if $AA^* = A^*A=I_{M_n(\mathbb{H})}$. Notice that the adjacency matrix of a $U(\mathbb{H})$-gain graph is a Hermitian matrix in $M_n(\mathbb{H})$.\\

For a square matrix $A\in M_n(\mathbb{H})$, we say that $q \in \mathbb{H}$ is a \emph{right eigenvalue} if there is a nonzero column vector ${\bf x} = (q_1,\ldots, q_n)^T$ such that $A{\bf x} = {\bf x}q$.
Notice that if $q\in \mathbb{H}$ is a right eigenvalue, then any $q'\in [q]$ is a right eigenvalue. Indeed, if $q' = h^{-1}qh$ and $A{\bf x} = {\bf x}q$, then the column vector ${\bf x}h$ satisfies the equation $A({\bf x}h)= ({\bf x}h)q'$.

Spectral theory of quaternionic matrices is a topic studied in the literature \cite{brenner,Lee,Spec,zhang}, also in relation with gain graphs  \cite{quat}. It has even been shown that the analogue of the spectral theorem holds.

\begin{theorem}\cite[Theorem 3.3, Proposition 3.8]{Spec}\label{Spectral Theorem}
If $A \in M_n(\mathbb{H})$ is normal then there are matrices $D, U\in M_n(\mathbb{H})$ such that
\begin{itemize}
    \item $U$ is a unitary matrix, $D$ is a diagonal matrix, and $U^*AU = D$;
    \item each diagonal entry of $D$ is a complex number $\lambda$ such that $Im(\lambda) \geq 0$;
    \item $q\in \mathbb{H}$ is a right eigenvalue of $A$ if and only if $[q] = [\lambda]$ for some diagonal element $\lambda$ of $D$.
\end{itemize}
Moreover, if $A=A^*$, the matrix $D$ is real.
\end{theorem}

It follows from Theorem \ref{Spectral Theorem} that with a Hermitian matrix $A\in M_n(\mathbb{H})$ a real diagonal matrix $D$ is associated.
The \emph{right spectrum} of $A$, that we denote by $\sigma_r(A)$, is the real multiset of elements on the main diagonal of $D$.
In particular, two Hermitian matrices are \emph{right cospectral} if and only if they are conjugated with each other via a unitary matrix. Clearly two $U(\mathbb{H})$-gain graphs are right cospectral if and only if their adjacency matrices are.

Observe that each quaternion $q = a + bi + cj + dk$ can be rewritten as
$$
q = (a + bi) + (c+di)j,
$$
Similarly, every quaternionic matrix $A\in M_n(\mathbb{H})$ can be decomposed as $A = A_1 + A_2j$, with $A_1, A_2\in M_n(\mathbb{C})$. Then the \emph{complex adjoint matrix} of $A$ is the matrix $f(A) \in M_{2n}(\mathbb{C})$ defined by
\begin{eqnarray*}
f(A) = \left(
         \begin{array}{cc}
           A_1 & A_2 \\
           -\overline{A}_2 & \overline{A}_1 \\
         \end{array}
       \right).
\end{eqnarray*}
The following properties of the complex adjoint matrix hold (see, for instance, \cite{Spec, Lee}).
\begin{proposition}\label{prop:f}
For each $A,B\in M_{n}(\mathbb{H})$ the following properties hold:
\begin{itemize}
\item $f(I_{M_n(\mathbb{H})}) = I_{2n}$;
\item $f(A + B) = f(A) + f(B)$;
\item $f(A B) = f(A)f(B);$
\item $f(A^*) = f(A)^*$;
\item $f(A^{-1}) = f(A)^{-1}$, when $A^{-1}$ exists.
\end{itemize}
\end{proposition}
These properties turn out to be very useful in order to investigate the right spectrum of a quaternionic matrix.

\begin{proposition}\label{f=r}
Let $A \in M_{n}(\mathbb{H})$, with $A^*=A$. Then  the spectrum of $f(A)$ consists exactly of two copies of the right spectrum of $A$. In particular, two Hermitian matrices $A$ and $B$ are right cospectral if and only if $f(A)$ and $f(B)$ are cospectral.
\end{proposition}
\begin{proof}
Let $A \in M_{n}(\mathbb{H})$, with $A^*=A$. Then, by Theorem \ref{Spectral Theorem}, there exist a real, diagonal matrix $D$ and a unitary matrix $U$ such that $A=UDU^*$. By Proposition \ref{prop:f}, we have $f(A)=f(U)f(D)f(U)^*$:
but $f(D)$ is a diagonal matrix where each diagonal element of $D$ appears twice: the statement follows.
\end{proof}

Another interesting consequence of Proposition \ref{prop:f} is that the restriction of the map $f$ to $M_1(\mathbb{H})$ induces a unitary representation of $U(\mathbb H)$, that we are going to introduce in the next definition.
\begin{definition}\label{def:piH}
We denote by $\pi_\mathbb H\colon U(\mathbb H)\to GL_2(\mathbb C)$ the unitary representation of degree $2$ of $U(\mathbb H)$ such that:
\begin{equation*}
\pi_\mathbb{ H}(1) =\left(
             \begin{array}{cc}
               1 & 0 \\
               0 & 1 \\
             \end{array}
           \right) \quad \pi_\mathbb{ H}( i) = \left(
             \begin{array}{cc}
               i & 0 \\
               0 & -i\\
             \end{array}
           \right)\quad \pi_\mathbb{ H}( j) = \left(
             \begin{array}{cc}
               0 & 1 \\
               -1 & 0 \\
             \end{array}
           \right) \quad  \pi_\mathbb{ H}( k) = \left(
             \begin{array}{cc}
               0 & i \\
               i & 0 \\
             \end{array}
           \right).
\end{equation*}
\end{definition}

We are now in position to present  a Godsil-McKay switching for the right spectrum of quaternion unit gain graphs.\\
Let $(\Gamma,\psi)$  be a $U(\mathbb{H})$-gain graph, and let $\alpha=\{C_0,C_1,\ldots, C_k\}$ be a partition of the vertex set $V_\Gamma$.
For any $v\in V_\Gamma$ and for each $i\in \{0,1,\ldots, k\}$, we set
$$
\Psi_i^{\mathbb{H}}(v):=\sum_{w\in C_i,\; w\sim v} \psi(v,w),
$$
where the sum is in  $\mathbb H$. Also, notice that $-1$ is a central involution in $U(\mathbb{H})$.
\begin{definition}\label{def:Qpartition}
A partition $\alpha=\left\{C_0,C_1,\ldots, C_k \right\}$ of the vertex set $V_\Gamma$ of a gain graph $(\Gamma,\psi)$  is said to be a \emph{quaternionic GM partition} if:
\begin{itemize}
\item
for every $i,j\in \{1,2,\ldots, k\}$ and for every $v,v'\in C_i$, one has
$$
\Psi_j^\mathbb{H}(v)=\Psi_j^\mathbb{H}(v');
$$
\item
for every $v\in C_0$, for every $i\in \{1,2,\ldots, k\}$, at least one of the following conditions holds:
\begin{enumerate}
\item there exist $q_1,q_2\in U(\mathbb{H}) \cup \left\{0\right\}$ such that $v$ is adjacent to half of the vertices of $C_i$ with gain $q_1$ and to the other half with gain $q_2$ (with the usual convention that the zero gain corresponds to no adjacency);
\item $\Psi_i^\mathbb{H}(v)=0$.
\end{enumerate}
\end{itemize}
\end{definition}

\begin{example}\label{ex:Q1}
Let us consider the $U(\mathbb H)$-gain graph $(\Gamma, \psi)$ depicted on the top of Fig.~\ref{fig:Qua}, and its partition $\alpha:=\{C_0,C_1\}$ with $C_0=\{v_7,v_8\}$ and $C_1=\{v_1,\ldots, v_6\}$.

We have
\begin{equation*}
\begin{split}
\Psi_1^{\mathbb{H}}(v_i)&=i-i=0\qquad \mbox{ for } i=1,2,3,4;\\
\Psi_1^{\mathbb{H}}(v_i)&=0  \qquad \mbox{ for } i=5,6;\\
\Psi_1^{\mathbb{H}}(v_7)&=\frac{\sqrt{2}}{2}+\frac{\sqrt{2}}{2}j+\frac{\sqrt{2}}{2}-\frac{\sqrt{2}}{2}j-\frac{\sqrt{2}}{2}+\frac{\sqrt{2}}{2}k-\frac{\sqrt{2}}{2}-\frac{\sqrt{2}}{2}k=0,
 \end{split}
\end{equation*}
and there is an edge with gain $1$ connecting $v_8$ to exactly half of the vertices of $C_1$. As a consequence, $\alpha$ is a quaternionic GM partition, according to Definition \ref{def:Qpartition}.

\begin{figure}
\centering

\begin{tikzpicture}[scale=1.4]


\node[vertex]  (1) at (0,0){\tiny $v_1$};

\node[vertex]  (2) at (0,2){\tiny $v_2$};

\node[vertex]  (3) at (2,2){\tiny $v_3$};

\node[vertex]  (4) at (2,0){\tiny $v_4$};

\node[vertex]  (7) at (4,0){\tiny $v_7$};

\node[vertex]  (5) at (5.3,1.3){\tiny $v_5$};

\node[vertex]  (6) at (6,0){\tiny $v_6$};

\node[vertex]  (8) at (1,1){\tiny $v_8$};

\node  (10) at (-1,1) {\small $(\Gamma,\psi)$};


\draw[dedge] (1) -- (2)node[midway, above,sloped] {\tiny $i$};
\draw[dedge] (2) -- (3)node[midway, above,sloped] {\tiny $i$};
\draw[dedge] (3) -- (4)node[midway, above,sloped] {\tiny $i$};
\draw[dedge] (4) -- (1)node[midway, below,sloped] {\tiny $i$};

\draw[edge] (8) -- (2);
\draw[edge] (8) -- (1);
\draw[edge] (8) -- (4);

\draw[dedge] (7) -- (4)node[midway, above, sloped] {\tiny $\frac{\sqrt{2}}{2}+\frac{\sqrt{2}}{2} j$};

\draw[dedge] (3) -- (7)node[midway, above, sloped] {\tiny $\frac{\sqrt{2}}{2}+\frac{\sqrt{2}}{2} j$};

\draw[dedge] (7) -- (5)node[midway, above, sloped] {\tiny $\frac{-\sqrt{2}}{2}+\frac{\sqrt{2}}{2} k$};

\draw[dedge] (6) -- (7)node[midway, above, sloped] {\tiny $\frac{-\sqrt{2}}{2}+\frac{\sqrt{2}}{2} k$};

\end{tikzpicture}\\
\begin{tikzpicture}[scale=1.4]


\node[vertex]  (1) at (0,0){\tiny $v_1$};

\node[vertex]  (2) at (0,2){\tiny $v_2$};

\node[vertex]  (3) at (2,2){\tiny $v_3$};

\node[vertex]  (4) at (2,0){\tiny $v_4$};

\node[vertex]  (7) at (4,0){\tiny $v_7$};

\node[vertex]  (5) at (5.3,1.3){\tiny $v_5$};

\node[vertex]  (6) at (6,0){\tiny $v_6$};

\node[vertex]  (8) at (6,2){\tiny $v_8$};

\node  (10) at (-1,1) {\small  $(\Gamma,\psi)^{\alpha, \mathbb H}$};


\draw[dedge] (1) -- (2)node[midway, above,sloped] {\tiny $i$};
\draw[dedge] (2) -- (3)node[midway, above,sloped] {\tiny $i$};
\draw[dedge] (3) -- (4)node[midway, above,sloped] {\tiny $i$};
\draw[dedge] (4) -- (1)node[midway, below,sloped] {\tiny $i$};

\draw[edge] (8) -- (3);
\draw[edge] (8) -- (5);
\draw[edge] (8) -- (6);

\draw[dedge] (7) -- (4)node[midway, above, sloped] {\tiny $\frac{-\sqrt{2}}{2}-\frac{\sqrt{2}}{2} j$};

\draw[dedge] (3) -- (7)node[midway, above, sloped] {\tiny $\frac{-\sqrt{2}}{2}-\frac{\sqrt{2}}{2} j$};

\draw[dedge] (7) -- (5)node[midway, above, sloped] {\tiny $\frac{\sqrt{2}}{2}-\frac{\sqrt{2}}{2} k$};

\draw[dedge] (6) -- (7)node[midway, above, sloped] {\tiny $\frac{\sqrt{2}}{2}-\frac{\sqrt{2}}{2} k$};

\end{tikzpicture}
\caption{The $U(\mathbb H)$-gain graphs $(\Gamma,\psi)$ and $(\Gamma,\psi)^{\alpha, \mathbb H}$ from Examples \ref{ex:Q1} and \ref{ex:Q2}.}\label{fig:Qua}
\end{figure}
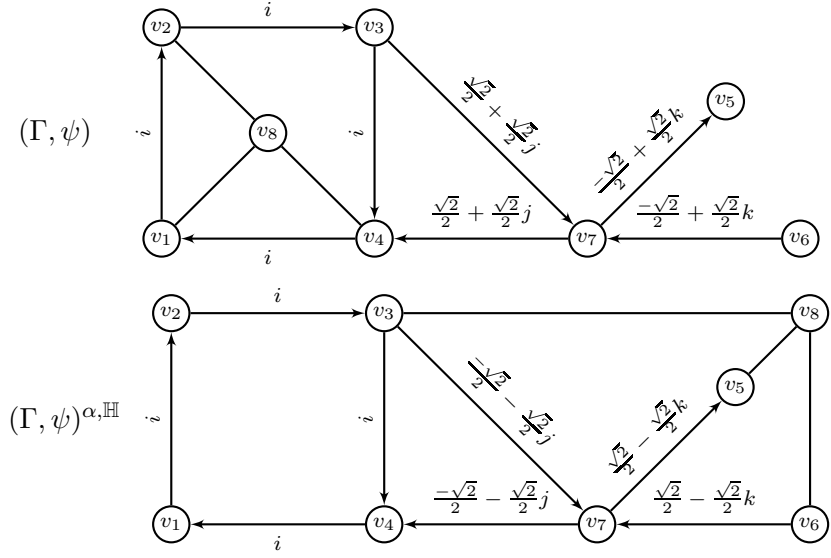
\end{example}

\begin{definition}
Let $\alpha$ be a quaternionic GM partition of the vertex set $V_\Gamma$ of $(\Gamma,\psi)$. Then the gain graph $(\Gamma,\psi)^{\alpha,\mathbb H}=(\Gamma^{\alpha,\mathbb H},\psi^{\alpha,\mathbb H})$ is defined as follows:
\begin{itemize}
\item the adjacency and the gains between pairs of vertices in $\bigcup_{i=1}^k C_i$ are the same as $(\Gamma,\psi)$;
\item for $v\in C_0$ and $i\in \{1,2,\ldots, k\}$, if we are in the case $(1)$ then we set:
$$
\psi^{\alpha,\mathbb H} (v,w)=
\begin{cases}
q_1 &\mbox{ if } \psi(v,w)=q_2,\\
q_2 &\mbox{ if } \psi(v,w)=q_1,
\end{cases}
$$
with the convention that $\psi(v,w)=0$ (resp. $\psi^{\alpha, \mathbb H}(v,w)=0$) means that $v$ and $w$ are not adjacent in $\Gamma$ (resp. in $\Gamma^{\alpha,\mathbb H}$).
If we are not in the case $(1)$ and we are in the case $(2)$, so that  $\Psi^{\mathbb H}_i(v)=0$, then we set
$$
\psi^{\alpha,\mathbb H}(v,w)=-\; \psi(v,w).
$$
\end{itemize}
The quaternionic GM partition $\alpha$ is \emph{nontrivial} if $(\Gamma,\psi)$ and  $(\Gamma,\psi)^{\alpha,\mathbb H}$ are not switching isomorphic.
\end{definition}

\begin{example}\label{ex:Q2}
In Fig.~\ref{fig:Qua} the $U(\mathbb H)$-gain graphs $(\Gamma,\psi)$ and  $(\Gamma,\psi)^{\alpha,\mathbb H}$, associated with the partition $\alpha$ described in Example \ref{ex:Q1} are depicted. Notice that $(\Gamma,\psi)$ and  $(\Gamma,\psi)^{\alpha,\mathbb H}$ are not switching isomorphic, since the underlying graphs $\Gamma$ and $\Gamma^{\alpha,\mathbb H}$ are clearly nonisomorphic. Then $\alpha$ is a nontrivial quaternionic GM partition of
$(\Gamma,\psi)$.
\end{example}

\begin{theorem}\label{teo:qfinale}
Let $(\Gamma, \psi)$ be a $U(\mathbb{H})$-gain graph and let $\alpha$ be a quaternionic GM partition. Then $(\Gamma, \psi)$ and $(\Gamma,\psi)^{\alpha,\mathbb H}$ are right cospectral.
\end{theorem}
\begin{proof}
As a first step we show that $\alpha$ is a quaternionic GM partition for $(\Gamma,\psi)$ if and only if it is  a $\pi_{\mathbb H}$-GM partition, where $\pi_{\mathbb H}$ is the unitary representation of $U(\mathbb H)$ of Definition \ref{def:piH}.
As in Eq.~\eqref{eq:grado} we have $\Psi_j(v)=\sum\limits_{w\in C_j, w\sim v} \psi(v,w)\in \mathbb C U(\mathbb H)$. One can easily check that
$$
\pi_{\mathbb H}(\Psi_j(v)) = \pi_{\mathbb H}(\Psi_j(v')) \iff \Psi_j^\mathbb{H}(v) = \Psi_j^\mathbb{H}(w)
$$
and that $$\pi_{\mathbb H}(\Psi_j(v))=0  \iff \Psi_j^\mathbb{H}(v)=0.$$
As a consequence,  the properties of Definition \ref{def:updated} hold for $\alpha$ if and only if the properties of Definition \ref{def:Qpartition} hold, and then a quaternionic GM partition is a   $\pi_{\mathbb H}$-GM partition. Moreover, the graph $(\Gamma,\psi)^{\alpha,\mathbb H}$ is exactly the $U(\mathbb H)$-gain graph $(\Gamma,\psi)^{\alpha,-1}$  of Definition \ref{def:consommanulla} with respect to $\pi_{\mathbb H}$.

By Theorem \ref{lemma-Q-pi-sommanulla}, the gain graphs $(\Gamma,\psi)$  and $(\Gamma,\psi)^{\alpha,\mathbb H}$ are $\pi_{\mathbb H}$-cospectral.
As a second step, we prove that the $\pi_{\mathbb H}$-spectrum is nothing but the spectrum of the complex adjoint matrix.

Let $A\in M_n(\mathbb H)$. We know that for each entry $a_{r,s}$ of $A$ there are two complex numbers $a_{r,s}^{1}$ and $a_{r,s}^{2}$ such that
$$
a_{r,s} = a_{r,s}^{1} + a_{r,s}^{2}j.
$$
Now, by applying Definition \ref{def:piH}, one can check that $\pi_{\mathbb H}(a_{r,s}) = \left(
                                                                                            \begin{array}{cc}
                                                                                              a_{r,s}^1 & a_{r,s}^2 \\
                                                                                              -\overline{a}_{r,s}^{2} &  \overline{a}_{r,s}^{1}  \\
                                                                                            \end{array}
                                                                                          \right)
$ and so
$$
\pi_{\mathbb H}(A) = \left(
  \begin{array}{cc|c|cc}
  a_{1,1}^{1} & a_{1,1}^{2} & \cdots  & a_{1,n}^{1} & a_{1,n}^{2} \\
-\overline{a}_{1,1}^{2} & \overline{a}_{1,1}^{1}  & \cdots & -\overline{a}_{1,n}^{2} & \overline{a}_{1,n}^{1}\\  \hline
\vdots & \vdots&   & \vdots& \vdots\\    \hline
a_{n,1}^{1} & a_{n,1}^{2} & \cdots & a_{n,n}^{1} & a_{n,n}^{2} \\ -\overline{a}_{n,1}^{2} & \overline{a}_{n,1}^{1} & \cdots & -\overline{a}_{n,n}^{2} & \overline{a}_{n,n}^{1}\\
  \end{array}
\right).
$$
Let us denote by $\tau:\{1, \ldots, 2n\} \longrightarrow \{1, \ldots, 2n\}$ the permutation defined as
$$
\tau(r) = \begin{cases} n + k \mbox{ if   } r=2k \\
    k+1 \mbox{ if   } r=2k +1.
    \end{cases}
$$
Let $e_b$ be the $2n$-vector which has $1$ in the $b$-th entry and $0$ otherwise. Let us denote by $C_\tau$ the permutation matrix of size $2n$ whose $b$-th column is $e_{\tau^{-1}(b)}$, and let $R_\tau$ be its transpose. Then, by using a shuffling approach as in \cite{shuffling}, one obtains:
$$
f(A) = R_\tau \pi_{\mathbb H}(A) C_\tau.
$$
Indeed
$$
\pi_{\mathbb H}(A) C_\tau =
\left(
  \begin{array}{cccc|cccc}
 a_{1,1}^{1} & a_{1,2}^{1} & \cdots & a_{1, n}^{1} & a_{1,1}^{2} & \cdots & a_{1,n-1}^{2} & a_{1,n}^{2}\\
    -\overline{a}_{1,1}^{2} & -\overline{a}_{1,2}^{2} & \cdots & -\overline{a}_{1, n}^{2} & \overline{a}_{1,1}^{1} & \cdots & \overline{a}_{1,n-1}^{1} & \overline{a}_{1,n}^{1}\\   \hline
    \vdots & \vdots & & \vdots & \vdots & & \vdots & \vdots \\      \hline
    a_{n,1}^{1} & a_{n,2}^{1} & \cdots & a_{n, n}^{1} & a_{n,1}^{2} & \cdots & a_{n,n-1}^{2} & a_{n,n}^{2}\\
    -\overline{a}_{n,1}^{2} & -\overline{a}_{n,2}^{2} & \cdots & -\overline{a}_{n, n}^{2} & \overline{a}_{n,1}^{1} & \cdots & \overline{a}_{n,n-1}^{1} & \overline{a}_{n,n}^{1}
  \end{array}
\right)
$$
and
$$
R_\tau \pi_{\mathbb H}(A) C_\tau =
\left(
  \begin{array}{cccc|cccc}
  a_{1,1}^{1} & a_{1,2}^{1} & \cdots & a_{1, n}^{1} & a_{1,1}^{2} & \cdots & a_{1,n-1}^{2} & a_{1,n}^{2}\\
    a_{2,1}^{1} & a_{2,2}^{1} & \cdots & a_{2, n}^{1} & a_{2,1}^{2} & \cdots & a_{2,n-1}^{2} & a_{2,n}^{2}\\
    \vdots & \vdots & &\vdots &\vdots & &\vdots & \vdots \\
    a_{n,1}^{1} & a_{n,2}^{1} & \cdots & a_{n, n}^{1} & a_{n,1}^{2} & \cdots & a_{n,n-1}^{2} & a_{n,n}^{2}\\    \hline
    -\overline{a}_{1,1}^{2} & -\overline{a}_{1,2}^{2} & \cdots & -\overline{a}_{1, n}^{2} & \overline{a}_{1,1}^{1} & \cdots & \overline{a}_{1,n-1}^{1} & \overline{a}_{1,n}^{1}\\
    \vdots & \vdots & &\vdots &\vdots & &\vdots & \vdots \\
   -\overline{a}_{n-1,1}^{2} & -\overline{a}_{n-1,2}^{2} & \cdots & -\overline{a}_{n-1, n}^{2} & \overline{a}_{n-1,1}^{1} & \cdots & \overline{a}_{n-1,n-1}^{1} & \overline{a}_{n-1,n}^{1}\\
    -\overline{a}_{n,1}^{2} & -\overline{a}_{n,2}^{2} & \cdots & -\overline{a}_{n, n}^{2} & \overline{a}_{n,1}^{1} & \cdots & \overline{a}_{n,n-1}^{1} & \overline{a}_{n,n}^{1}
  \end{array}
\right),
 $$
which is exactly the matrix $f(A)$. Notice that, since $R_\tau$ is a permutation matrix, then $R_\tau^{-1} = C_\tau$ and so $\pi_{\mathbb H}(A)$ and $f(A)$ are similar. This implies that they have the same eigenvalues with the same multiplicities. \\
\indent As a consequence,  the fact that $(\Gamma,\psi)$  and $(\Gamma,\psi)^{\alpha,\mathbb H}$ are $\pi_{\mathbb H}$-cospectral implies that the complex adjoint matrices
$f(A_{(\Gamma,\psi)})$ and $f(A_{(\Gamma,\psi)^{\alpha,\mathbb H}})$ are cospectral: by virtue of Proposition \ref{f=r} it is possible to conclude that
$(\Gamma,\psi)$  and $(\Gamma,\psi)^{\alpha,\mathbb H}$ are right cospectral.
\end{proof}

\begin{example}\label{ex:Q3}
Consider the $U(\mathbb H)$-gain graphs $(\Gamma,\psi)$  and $(\Gamma,\psi)^{\alpha,\mathbb H}$ of Examples \ref{ex:Q1} and \ref{ex:Q2}.
An explicit computation shows that the characteristic polynomials of the matrices $f(A_{(\Gamma,\psi)})$ and  $f(A_{(\Gamma,\psi)^{\alpha,\mathbb H}})$, or equivalently, the characteristic polynomials of the matrices
$\pi_{\mathbb H}(A_{(\Gamma,\psi)})$ and  $\pi_{\mathbb H}(A_{(\Gamma,\psi)^{\alpha,\mathbb H}})$, coincide and are equal to
$$
x^{16}-22x^{14}+187x^{12}-776x^{10}+1639x^8-1650x^6+625x^4
$$
and so the  $U(\mathbb H)$-gain graphs $(\Gamma,\psi)$ and $(\Gamma,\psi)^{\alpha,\mathbb H}$ are right cospectral.
\end{example}

\end{document}